\documentclass[12pt,reqno]{amsart}

\usepackage{amsfonts}
\usepackage{amscd}
\usepackage{amssymb}
\usepackage{amsfonts}
\usepackage{amscd}
\usepackage{amsmath} 
\usepackage{mathrsfs}
\usepackage{dsfont}
\usepackage{enumerate}
\usepackage{xcolor}
\usepackage{float} % To use [H]
\usepackage[pdftex]{graphicx}
\allowdisplaybreaks
\usepackage{url}
\usepackage[colorlinks=true,
            linkcolor=blue!70!black,
            citecolor=teal,
            urlcolor=blue!60!black]{hyperref}

\date{\today}

\newtheorem*{theorem*}{Theorem}
\newtheorem{theorem}{Theorem}[section]
\newtheorem{corollary}[theorem]{\bf{Corollary}}
\newtheorem{lemma}[theorem]{Lemma}
\newtheorem{proposition}[theorem]{Proposition}
\theoremstyle{definition}
\newtheorem{definition}[theorem]{Definitions}
\theoremstyle{remark}
\newtheorem{remark}[theorem]{\bf{Remark}}

\numberwithin{equation}{section}

\newcommand{\beas}{\begin{eqnarray*}}
\newcommand{\eeas}{\end{eqnarray*}}
\newcommand{\bes} {\begin{equation*}}
\newcommand{\ees} {\end{equation*}}
\newcommand{\be} {\begin{equation}}
\newcommand{\ee} {\end{equation}}
\newcommand{\bea} {\begin{eqnarray}}
\newcommand{\eea} {\end{eqnarray}}

\newcommand{\R}{\mathbb R}

\newcommand{\bn}{\mathbb{B}^N}

\usepackage{color}

\title[Multiple Solutions to the Brezis-Nirenberg Problem on Hyperbolic Spaces]{Multiplicity of Solutions to the Brezis-Nirenberg Problem on Hyperbolic Spaces}

\author[Ghosh, Kumar, Rana]{Sekhar Ghosh, Vishvesh Kumar,  Tapendu Rana}	

\address{Sekhar Ghosh \endgraf  Department of Mathematics, National Institute of Technology Calicut, 
\endgraf NITC PO, Kozhikode, Kerala, 673601, India.}
\email{sekharghosh1234@gmail.com, sekharghosh@nitc.ac.in}

\address{Vishvesh Kumar  \endgraf   Department of Mathematical Sciences,  Indian Institute of Technology (BHU), 
\endgraf Varanasi, Uttar Pradesh, 221005, India.}
\email{vishvesh.mat@iitbhu.ac.in, vishveshmishra@gmail.com,}

 \address{Tapendu Rana  \endgraf Department of Mathematics: Analysis, Logic and Discrete Mathematics,	\endgraf Ghent University, 	Krijgslaan 281, Building S8, B 9000 Ghent, Belgium.} \email{tapendurana@gmail.com, tapendu.rana@ugent.be}

\date{}
\begin{document}
\begin{abstract}
This article investigates the multiplicity of solutions to the Brezis-Nirenberg problem on smooth bounded domains in the hyperbolic space $\mathbb{B}^N$ for $N \ge 4$. Specifically, we study the critical semilinear equation $-\Delta_{\mathbb{B}^N} u = \lambda u + |u|^{2^*-2}u$ under Dirichlet boundary conditions for $\lambda > \frac{N(N-2)}{4}$. Overcoming the analytic challenges induced by the hyperbolic geometry and the intricate concentration profiles of Palais--Smale sequences, we establish the existence of multiple pairs of nontrivial solutions. Using the equivariant Ljusternik-Schnirelmann category, we obtain lower bounds on the number of solutions depending on the position of the parameter $\lambda$ relative to the Dirichlet spectrum of the Laplace-Beltrami operator.
\end{abstract}

\maketitle
\section{Introduction}

In this article, we study the existence of multiple solutions to the problem 
\begin{equation}\label{eqn_MBNP_Hn_int}\tag{$\wp$}
\begin{cases}
    -\Delta_{\bn} u = \lambda u + |u|^{2^*-2}u & \text{in } \Omega_{\bn}, \\
    u = 0 & \text{on } \partial\Omega_{\bn},
\end{cases}
\end{equation}
where  $\Omega_{\bn}$ is a smooth bounded domain in the Poincaré ball model of the hyperbolic space $\bn$,  $ -\Delta_{\bn}$ denotes the Laplace-Beltrami operator on $\bn$, $\lambda> \frac{N(N-2)}{4}$, $N \geq 4$, and $2^* = \frac{2N}{N-2}$ is the critical Sobolev exponent.

\subsection{Background on the Brezis-Nirenberg Problem}
The study of semi-linear elliptic equations with critical growth began with the seminal work of Brezis and Nirenberg \cite{BN83}, who investigated the Euclidean version of \eqref{eqn_MBNP_Hn_int} on bounded domains $\Omega \subset \mathbb{R}^N$:
\begin{equation}\label{eqn_MBNP_Rn_int}
\begin{cases}
    -\Delta u = \lambda u + |u|^{2^*-2}u & \text{in } \Omega, \\
    u = 0 & \text{on } \partial\Omega.
\end{cases}
\end{equation}
Historically, interest in these equations was motivated by their deep connections to the Yamabe problem in Riemannian geometry, as well as critical phenomena in physics such as the Yang-Mills functional and H-systems \cite{BN83, JL87}. 

There has been an extensive study of the Brezis-Nirenberg problem over the past few decades, particularly concerning the multiplicity of solutions. The first multiplicity results were obtained by Cerami, Fortunato, and Struwe \cite{CFS84}. Later, Devillanova and Solimini \cite{DS02} obtained profound multiplicity results by showing that if $N \geq 7$, then \eqref{eqn_MBNP_Rn_int} possesses infinitely many solutions for every $\lambda > 0$ (see also \cite{DS03, SZ10}). Building on these foundations, Clapp and Weth \cite{CW05} established significant multiplicity results in Euclidean space by exploiting the topology of the solution space. They established that for $N \geq 4$, if $\lambda > 0$ is not a Dirichlet eigenvalue, the problem admits at least $(N+1)/2$ pairs of nontrivial solutions. Furthermore, they showed that if $0 < \lambda < \Lambda_1$ (where $\Lambda_1$ is the first Dirichlet eigenvalue of $-\Delta$ on $\Omega$), the number of nontrivial solution pairs is at least $(N+2)/2$ and provided lower bounds even when $\lambda$ is an eigenvalue of multiplicity $m$. For further progress on the study of the Brezis-Nirenberg problem in the Euclidean settings, we refer to \cite{AGKR25, SV13, SV15, AU22, SFV24} and the references therein. 

\subsection{The Brezis-Nirenberg Problem in Hyperbolic Spaces}
The investigation of the Brezis-Nirenberg problem on negative curvature manifolds, specifically on bounded domains within the hyperbolic space $\bn$, was initiated by Stapelkamp \cite{Sta02, Sta03}. Let $\lambda_n$ denote the $n$-th Dirichlet eigenvalue of the Laplace-Beltrami operator $-\Delta_{\bn}$ on $\Omega_{\bn}$, counted with multiplicity. Stapelkamp's foundational result can be stated as follows:

\begin{theorem}[{{\cite[Page 48]{Sta03}}}]\label{thm_BN_Hn}
Let $\Omega_{\bn}$ be a bounded domain in $\mathbb{B}^N$ for $N \geq 4$. Then the following statements hold:
\begin{enumerate}
    \item\label{statment_1} For $\lambda \geq \lambda_1$, problem \eqref{eqn_MBNP_Hn_int} admits no nontrivial positive solutions.
    \item\label{statment_2} For $\lambda \leq \frac{N(N-2)}{4}$, if $\Omega_{\bn}$ is star-shaped, problem \eqref{eqn_MBNP_Hn_int} admits no nontrivial positive solutions.
    \item If $\frac{N(N-2)}{4} < \lambda < \lambda_1$, then problem \eqref{eqn_MBNP_Hn_int} admits at least one nontrivial positive solution.
\end{enumerate}
\end{theorem} 
\begin{remark}
An analogous version of statements \eqref{statment_1} and \eqref{statment_2} in Theorem \ref{thm_BN_Hn} remains true for $N=3$; we refer to \cite[Page 59]{Sta03} for more details. For the precise definition of a star-shaped domain on $\bn$, we refer the reader to \cite[Page 13]{Sta03}.
\end{remark}
\noindent More recently, Li, Lu, and Yang  \cite{LLY22} have expanded this theory to the higher-order Brezis-Nirenberg problem associated with GJMS operators on both bounded domains and the entire hyperbolic space, utilizing Helgason-Fourier analysis and Green's function estimates.

Parallel research has extensively explored these problems on the full hyperbolic space. Mancini and Sandeep \cite{MS08} provided foundational results on the existence, nonexistence, and uniqueness of positive entire solutions, proving that a solution on the whole of $\bn$ exists if and only if $\frac{N(N-2)}{4} < \lambda < \frac{(N-1)^2}{4}$ for $N \geq 4$. Ganguly and Sandeep \cite{GK14} extended this work to sign-changing solutions, establishing that no solutions exist for $\lambda \leq \frac{N(N-2)}{4}$ and proving the existence of infinitely many radial sign-changing solutions when $N \geq 7$. Very recently, Bhakta, Ganguly, Gupta, and Sahoo \cite{BGGS25} analyzed the global compactness properties and profile decompositions for critical problems in the hyperbolic setting, highlighting how concentration occurs via both hyperbolic isometries and localized Aubin--Talenti bubbles. An excellent survey of recent developments in the study of PDEs on hyperbolic spaces is presented in Sandeep \cite{San26}. For related progress on the sub-Riemannian version of the Brezis-Nirenberg problem, we refer the reader to recent works \cite{GKR24a, GKR24b, GK25a, L07, L19} and the references therein.

\subsection{Main Result} 
Despite the extensive literature on existence and compactness profiles in hyperbolic space, the \textit{multiplicity} of solutions for the Brezis-Nirenberg problem on bounded hyperbolic domains has remained largely unexplored. 

{In this article, our primary focus is to fill this gap by establishing the existence of multiple pairs of solutions for the hyperbolic Brezis-Nirenberg problem within smooth bounded domains for all parameters $\lambda > \frac{N(N-2)}{4}$. While our topological strategy draws inspiration from the Euclidean framework introduced by Clapp and Weth \cite{CW05}, the hyperbolic geometry introduces severe analytic obstacles. The negative curvature of the space induces a critical spectral shift $\frac{N(N-2)}{4}$ that fundamentally compromises the standard coercivity of the operator. Furthermore, the conformal weight of the Riemannian metric intricately alters the concentration profiles and compactness thresholds of the associated Palais--Smale sequences. Crucially, the presence of this spatial weight completely breaks the translation invariance of the space. Consequently, standard Euclidean techniques of spatially shifting localized Aubin--Talenti bubbles fail, as such translations pull the test functions off the Nehari manifold and destroy the strict odd symmetries required for topological multiplicity arguments.}

{To overcome these geometric obstacles, we must finely calibrate the relative equivariant Ljusternik-Schnirelmann category to our weighted Euclidean reduction. By systematically employing continuous Nehari retractions, we correct the broken translation invariance and restore the necessary $\mathbb{Z}_2$-symmetries. This framework allows us to accurately track the min-max energy levels against the shifted hyperbolic spectrum and explicitly separate the topological critical points from the Aubin--Talenti concentration levels. Our main result is as follows:
}
\begin{theorem}\label{th_MBNP_Hn}
    Assume $N \geq 4$.
\begin{enumerate}
    \item[(i)]  If $\lambda_n < \lambda < \lambda_{n+1}$, then problem \eqref{eqn_MBNP_Hn_int} admits at least $\frac{N+1}{2}$ pairs of nontrivial solutions.
    \item[(ii)] If $\frac{N(N-2)}{4} < \lambda < \lambda_1$, then problem \eqref{eqn_MBNP_Hn_int} admits at least $\frac{N+2}{2}$ pairs of nontrivial solutions.
    \item[(iii)] If $\lambda = \lambda_{n+1} = \cdots = \lambda_{n+m}$ is an eigenvalue of multiplicity $m < N+2$, then problem \eqref{eqn_MBNP_Hn_int} has at least $\frac{N+1-m}{2}$ pairs of nontrivial solutions.
\end{enumerate}
Furthermore, these solutions satisfy the uniform energy estimate
\begin{align}\label{eq_nabu2<2S}
\int_{\Omega_{\bn}} |\nabla_{\bn} u|^2 \, dV - \frac{N(N-2)}{4} \int_{\Omega_{\bn}} |u|^{2} \, dV < 2 S^{N / 2},
\end{align}
where $S$ is the best constant for the Euclidean Sobolev embedding $D^{1,2}(\mathbb{R}^N) \hookrightarrow L^{2^*}(\mathbb{R}^N)$.
\end{theorem}
 \subsection*{Outline of the Article:}
The remainder of this article is organised as follows. 
In Section \ref{sec_prelim}, we establish the geometric framework of the Poincar{\'e} ball model and implement a conformal change of variables to transform the hyperbolic problem into an equivalent weighted elliptic equation in Euclidean space. Section \ref{sec_grad_flow_cat} is dedicated to the variational setting, where we define the negative gradient flow and the relative equivariant Ljusternik–Schnirelmann category. In Section \ref{sec_palais_smale}, we provide a detailed analysis of the Palais–Smale decomposition, utilizing Struwe’s global compactness theorem to characterize the \textit{bubbling} behavior of concentrating sequences. Section \ref{sec_energy_est} develops the core energy estimates on the Nehari manifold and combines these results to establish the multiplicity theorem for the various spectral cases of $\lambda$.

\section{Geometric Framework and Conformal Reduction}\label{sec_prelim}
In this section, we establish the analytic and geometric foundations for studying the main problem. We begin by reviewing the Poincar{\'e} ball model of the hyperbolic space. Subsequently, we employ a conformal change of variables to transform the nonlinear hyperbolic equation into an equivalent weighted elliptic problem in the standard Euclidean setting. 

\subsection{Ball Model of Hyperbolic Spaces} 

Let $\bn := \{x \in \R^N : |x| < 1\}$ be the unit open ball in $\R^N$. The space $\bn$, endowed with the conformal Riemannian metric $g_{ij}(x) = \varrho^2(x) \delta_{ij}$, where
\begin{align*}
    \varrho(x) = \frac{2}{1 - |x|^2},
\end{align*}
is known as the Poincar\'e ball model of the $N$-dimensional hyperbolic space. Throughout this article, we will work extensively with this model (for more details, see, for instance, \cite{BP92, Rat94}). 

Using standard Euclidean coordinates, the hyperbolic line element and volume element are given respectively by
\begin{align*}
    ds^2 = g_{ij} \, dx^i dx^j = \varrho^2(x) |dx|^2
\end{align*}
and
\begin{align*}
    dV = \sqrt{\det(g_{ij})} \, dx = \varrho^N(x) \, dx,
\end{align*}
where $dx$ and $|dx|$ denote the standard Euclidean volume and line elements.

Consequently, for a given smooth bounded domain $\Omega_{\bn} \subset \bn$,  the hyperbolic gradient and   the  norm of the Dirichlet integral take the form
\begin{align*}
    \nabla_{\bn} u = \varrho^{-1}(x) \nabla u
\end{align*}
and
\begin{align*}
    \int_{\Omega_{\bn}} |\nabla_{\bn} u|^2 \, dV = \int_{\Omega} |\nabla u|^2 \varrho^{N-2}(x) \, dx.
\end{align*}
Furthermore, the Laplace-Beltrami operator on $\bn$ can be expressed in terms of the Euclidean Laplacian as
\begin{equation}\label{eq_reln_Delta_Hn_Rn}
\begin{aligned}
    \Delta_{\bn} u &= \frac{1}{\varrho^N(x)} \sum_{i,j=1}^N \frac{\partial}{\partial x^i} \left( \varrho^N(x) \, \varrho^{-2}(x)\delta^{ij} \frac{\partial u}{\partial x^j} \right) \\
    &= \varrho^{-N}(x) \nabla \cdot \big( \varrho^{N-2}(x) \nabla u \big) \\
    &= \varrho^{-2}(x) \Delta u + (N-2) \varrho^{-3}(x) \nabla \varrho \cdot \nabla u \\
    &= \left(\frac{1-|x|^2}{2}\right)^2 \Delta u + (N-2) \left( \frac{1-|x|^2}{2} \right) (x \cdot \nabla u).
\end{aligned}
\end{equation}
\begin{remark}
    Here, $\Omega_{\bn} \subset \mathbb{B}^N$ and $\Omega \subset \mathbf{B}(0,1) \subset \mathbb{R}^N$ represent the exact same geometric domain, where $\mathbf{B}(x,r) := \{\, y \in \mathbb{R}^N : |y - x| < r \,\}$ denotes the open Euclidean ball of radius $r$ centered at $x$. The distinct notation is used simply to emphasize when explicit Euclidean coordinates and measures are being employed.
\end{remark}

\subsection{Transformation of the Brezis-Nirenberg Problem on $\bn$ to the Euclidean Laplacian}

In this subsection, we transform the problem from the hyperbolic space into an equivalent problem in standard Euclidean space. We follow an approach similar to the one found in \cite[Appendix A.1]{Sta03}. For the sake of clarity and completeness, we present the detailed calculations here. 

Let $\Omega_{\bn} \subset \mathbb{B}^N$ ($N \geq 3$) be a bounded domain with a boundary that is sufficiently smooth to permit integration by parts. Let $\Omega \subset \textbf{B}(1,0) \subset \mathbb{R}^N$ denote the corresponding domain in Euclidean coordinates, and let $\lambda \in \mathbb{R}$ be a parameter.
 
Using \eqref{eq_reln_Delta_Hn_Rn}, the Brezis-Nirenberg type problem \eqref{eqn_MBNP_Hn_int} takes the following form in Euclidean coordinates:
\begin{equation}\label{eqn_BN_prob}
\begin{cases}
    \varrho^{-N} \nabla \cdot \big(\varrho^{N-2} \nabla u \big) + \lambda u + |u|^{2^* - 2}u = 0 & \text{in } \Omega, \\
    u = 0 & \text{on } \partial\Omega.
\end{cases}
\end{equation}
A weak solution $u \in H^1_0(\Omega)$ of \eqref{eqn_BN_prob} satisfies the integral identity
\begin{equation}\label{int_eqn_BN}
    \int_\Omega |\nabla u|^2 \varrho^{N-2} \, dx - \lambda \int_{\Omega} u^2 \varrho^N \, dx - \int_{\Omega} |u|^{2^*} \varrho^N \, dx = 0.
\end{equation}

To eliminate the weight $\varrho^{N-2}$ in the gradient term, we introduce the conformal change of variables 
\begin{align}\label{eq_conf_trans}
    v(x) := \varrho^{\frac{N-2}{2}}(x) u(x),
\end{align} 
for $v \in H^1_0(\Omega)$. Differentiating $u(x) = \varrho^{-\frac{N-2}{2}}(x) v(x)$ and using the identity $\nabla \varrho = \varrho^2 x$, we obtain
\begin{equation*}
    \nabla u = -\frac{N-2}{2} \varrho^{-\frac{N}{2}+2} v x + \varrho^{-\frac{N}{2}+1} \nabla v = \varrho^{-\frac{N}{2}+1} \left( \nabla v - \frac{N-2}{2} \varrho v x \right).
\end{equation*}
Expanding the squared norm yields
\begin{equation}\label{exp_|Gu|^2}
    |\nabla u|^2 = \varrho^{-N+2} \left( |\nabla v|^2 + \frac{(N-2)^2}{4} \varrho^2 v^2 |x|^2 - \frac{N-2}{2} \varrho x \cdot \nabla(v^2) \right).
\end{equation}
Multiplying \eqref{exp_|Gu|^2} by $\varrho^{N-2}$ and integrating over $\Omega$, we apply integration by parts to the last term. Noting the vector calculus identity $\nabla \cdot (\varrho x) = \varrho^2 |x|^2 + N \varrho$, and using the fact that $N \varrho = \frac{N}{2} \varrho^2 (1-|x|^2)$, we compute:
\begin{align*}
    \int_{\Omega} |\nabla u|^2 \varrho^{N-2} \, dx 
    &= \int_{\Omega} |\nabla v|^2 \, dx + \frac{(N-2)^2}{4} \int_{\Omega} v^2 |x|^2 \varrho^2 \, dx - \frac{N-2}{2} \int_{\Omega} \varrho \nabla(v^2) \cdot x \, dx \\
    &= \int_{\Omega} |\nabla v|^2 \, dx + \frac{(N-2)^2}{4} \int_{\Omega} v^2 |x|^2 \varrho^2 \, dx + \frac{N-2}{2} \int_{\Omega} v^2 \nabla \cdot (\varrho x) \, dx \\
    &= \int_{\Omega} |\nabla v|^2 \, dx + \frac{(N-2)^2}{4} \int_{\Omega} v^2 |x|^2 \varrho^2 \, dx \\
    &\hspace{2cm} + \frac{N-2}{2} \int_{\Omega} v^2 \varrho^2 \left( |x|^2  + \frac{N}{2} - \frac{N}{2}|x|^2 \right) \,dx \\
    &= \int_{\Omega} |\nabla v|^2 \, dx + \frac{(N-2)^2}{4} \int_{\Omega} v^2 |x|^2 \varrho^2 \, dx  
    + \frac{N(N-2)}{4} \int_{\Omega} v^2 \varrho^2 \, dx \\
     & \hspace{2cm} - \frac{(N-2)^2}{4} \int_{\Omega} v^2 \varrho^2 |x|^2 \, dx \\
    &= \int_{\Omega} |\nabla v|^2 \, dx + \frac{N(N-2)}{4} \int_{\Omega} v^2 \varrho^2 \, dx.
\end{align*}
Substituting $v = \varrho^{\frac{N-2}{2}} u$ back into the integral equation \eqref{int_eqn_BN}, the weights cancel perfectly, transforming the equation into
\begin{equation*}
    \int_{\Omega} |\nabla v|^2 \, dx + \frac{N(N-2)}{4} \int_{\Omega} v^2 \varrho^2 \, dx - \lambda \int_{\Omega} v^2 \varrho^{-N+2} \varrho^N \, dx - \int_{\Omega} |v|^{2^*} \varrho^{-N} \varrho^N \, dx = 0.
\end{equation*}
This can be equivalently rewritten as
\begin{align*}
    \int_{\Omega} |\nabla v|^2 \, dx + \left( \frac{N(N-2)}{4} - \lambda \right) \int_{\Omega} v^2 \varrho^2 \, dx - \int_{\Omega} |v|^{2^*} \, dx = 0.
\end{align*}
This demonstrates that $v = \varrho^{\frac{N-2}{2}} u \in H^1_0(\Omega)$ is a weak solution of the corresponding boundary value problem
\begin{equation}\label{eq_trns_BNP_v}
\begin{cases}
    -\Delta v = \left( \lambda -\frac{N(N-2)}{4}  \right) \varrho^2 v + |v|^{2^*-2} v & \text{in } \Omega, \\
    v = 0 & \text{on } \partial \Omega.
\end{cases}
\end{equation}
Therefore, the conformal transformation \eqref{eq_conf_trans} successfully converts the problem involving the Laplace-Beltrami operator on $\bn$ into an equivalent problem governed by the standard Euclidean Laplacian on $\R^N$. Consequently, it suffices to study the solutions of \eqref{eq_trns_BNP_v}. Henceforth, our primary objective is to establish the existence of multiple nontrivial solutions to this transformed problem.

The solutions of problem \eqref{eq_trns_BNP_v} correspond to the critical points of the $C^2$-functional $\mathcal{I}_{\lambda, \varrho} : H^1_0(\Omega) \to \mathbb{R}$ defined by
\begin{align}\label{de_of_I_l,p}
\mathcal{I}_{\lambda, \varrho}(v)
= \frac{1}{2} \int_{\Omega} \left( |\nabla v|^2 - \left( \lambda - \frac{N(N-2)}{4}\right) \varrho^2 v^2 \right) dx - \frac{1}{2^*} \int_{\Omega} |v|^{2^*} dx.   
\end{align}
To properly analyze the linear part of \eqref{eq_trns_BNP_v}, let us consider the associated weighted eigenvalue problem: find $\mu \in \mathbb{R}$ and a nontrivial function $v \in H_{0}^{1}(\Omega)$ satisfying
\begin{equation}\label{eqn_ev_pro}
\begin{cases}
    -\Delta v = \mu \varrho^2(x) v & \text{in } \Omega, \\
    v = 0 & \text{on } \partial \Omega.
\end{cases}
\end{equation}
Since $\Omega$ is strictly contained within the unit ball $\mathbf{B}(0,1)$, the conformal weight $\varrho$ is strictly bounded from below and above on $\Omega$. By setting the base spectral shift $\lambda_0 := \frac{N(N-2)}{4}$, the linear coefficient in our problem becomes precisely $\mu = \lambda - \lambda_0$. By standard spectral theory for elliptic operators with bounded weights (see, for instance, \cite{Cue01} and \cite[(1.4)]{Xua03}), problem \eqref{eqn_ev_pro} admits a discrete sequence of eigenvalues $0 < \mu_1 < \mu_2 \leq \mu_3 \leq \cdots \leq \mu_n \to \infty$. 

Consequently, the eigenvalues $\lambda_n$ of the hyperbolic Laplace-Beltrami operator $-\Delta_{\bn}$ correspond exactly to $\lambda_n = \mu_n + \lambda_0$. This yields the shifted sequence
\begin{align*}
    \frac{N(N-2)}{4}= \lambda_0 < \lambda_1 < \lambda_2 \leq \lambda_3 \leq \cdots \leq \lambda_n \to \infty.
\end{align*}
(For a comprehensive discussion on the first eigenvalue $\lambda_1$ of the Laplace-Beltrami operator on hyperbolic domains, we refer the reader to \cite[Chapter 4]{Sta03}.) 

 Using this spectral framework, our main geometric result Theorem \ref{th_MBNP_Hn} follows directly from establishing the following equivalent theorem in the Euclidean setting.
\begin{theorem}\label{th_MBNP_Hn>Rn}
    Assume $N \geq 4$.
\begin{enumerate}
    \item[(i)]  If $\lambda_n < \lambda < \lambda_{n+1}$, then problem \eqref{eq_trns_BNP_v} admits at least $\frac{N+1}{2}$ pairs of nontrivial solutions.
    \item[(ii)] If $\lambda_0 < \lambda < \lambda_1$, then problem \eqref{eq_trns_BNP_v} admits at least $\frac{N+2}{2}$ pairs of nontrivial solutions.
    \item[(iii)] If $\lambda = \lambda_{n+1} = \cdots = \lambda_{n+m}$ is an eigenvalue of multiplicity $m < N+2$, then problem \eqref{eq_trns_BNP_v} has at least $\frac{N+1-m}{2}$ pairs of nontrivial solutions.
\end{enumerate}
Furthermore, these solutions satisfy the Euclidean energy estimate
\begin{align}\label{eq_nabu2<2S_Rn}
\int_{\Omega} |\nabla v |^2 \, dx < 2 S^{N / 2}.   
\end{align}
%where $S$ is the best constant for the Sobolev embedding $D^{1,2}(\R^N)  \hookrightarrow L^{2^*}(\R^N)$. 
\end{theorem}

\section{Theory of Gradient Flow and Relative Equivariant Category}\label{sec_grad_flow_cat} 

In this section, we define and discuss the properties of the gradient flow associated with the functional $\mathcal{I}_{\lambda, \varrho}$ and the notion of the relative
equivariant Ljusternik–Schnirelmann category. Before doing so, we introduce some necessary notation and recall a few standard facts.

The Hilbert space $D^{1,2}(\mathbb{R}^N)$ is defined as the completion of $C_c^{\infty}(\mathbb{R}^N)$ with respect to the norm $\| \cdot\|$ induced by the scalar product
\[
\langle  v, \varphi \rangle := \int_{\mathbb{R}^N} \nabla v \cdot \nabla \varphi \, dx.
\]
For any subset $A \subset D^{1,2}(\mathbb{R}^N)$ and element $v \in D^{1,2}(\mathbb{R}^N)$, the distance from $v$ to $A$ is defined by
\[
\operatorname{dist}(v, A) := \inf_{\varphi \in A} \| v-\varphi\|.
\]
For $v \in L^p(\mathbb{R}^N)$, we denote by $|v|_p$ the usual $L^p$-norm. Let $\Omega \subset \mathbb{R}^N$ be a bounded smooth domain and define $H := H_0^1(\Omega) \subset D^{1,2}(\mathbb{R}^N)$ as the completion of the space $C^{\infty}_0(\Omega)$ (smooth functions with compact support in $\Omega$). 

For any $A \subset H$ and $\delta > 0$, we denote the closed $\delta$-neighborhood of $A$ by
\[
B_\delta(A) := \{v \in H : \operatorname{dist}(v, A) \leq \delta\},
\]
and we denote by $\operatorname{int}(A)$ the interior of $A$ in $H$.

We recall that $\lambda_0 = {N(N-2)}/{4}$. The principal shifted eigenvalue $\lambda_1 - \lambda_0$ of the corresponding weighted eigenvalue problem admits the following variational characterization (see \cite{Cue01}):
\begin{align}\label{eqn_prop_lam1}
    \lambda_1 - \lambda_0 = \inf_{\varphi \in H \setminus \{0\}} \frac{\int_{\Omega} |\nabla \varphi|^2 \, dx}{\int_{\Omega} \varrho^2 |\varphi|^2 \, dx}.
\end{align}
For a fixed $\lambda >\lambda_0$, we fix integers $n, m \in \mathbb{N} \cup \{0\}$ such that 
\[
\lambda_m < \lambda < \lambda_{m+k+1},
\]
where $n$ is the largest integer satisfying $\lambda_n < \lambda$, and $m$ is the smallest integer satisfying $\lambda < \lambda_{n+m+1}$. We consider a sequence of $H$-orthonormal eigenfunctions $\{e_k\}$ corresponding to the eigenvalues $\{\lambda_k-\lambda_0\}_{k \in \mathbb{N}}$ of the problem \eqref{eqn_ev_pro}, meaning they satisfy:
\begin{equation}\label{eqn_ev_pro_en}
\begin{cases}
    -\Delta e_k = \left( \lambda_k -\lambda_0 \right)\varrho^2(x) e_k & \text{in } \Omega, \\
    e_k = 0 & \text{on } \partial \Omega.
\end{cases}
\end{equation}
We define the associated finite- and infinite-dimensional subspaces as
\begin{align}\label{eqn_def_V-+}
V^{-} := \operatorname{span}\{e_1, \ldots, e_n\}, \qquad 
V^{+} := \operatorname{span}\{e_j : j > n+m\}.   
\end{align}
\subsection{Gradient Flow Associated with $\mathcal{I}_{\lambda, \varrho}$}
We recall the definition of the functional ${\mathcal{I}_{\lambda, \varrho}}$ from \eqref{de_of_I_l,p}. The action of its Fréchet derivative $\nabla {\mathcal{I}_{\lambda, \varrho}}$ on a test function $\varphi \in H$ is given by
\begin{align*}
    \langle \nabla {\mathcal{I}_{\lambda, \varrho}}(v), \varphi \rangle &= \left. \frac{d}{dt} \mathcal{I}_{\lambda, \varrho} (v+t \varphi) \right|_{t=0} \\
    &= \int_{\Omega} \nabla v \cdot \nabla \varphi \, dx - \left(\lambda - \lambda_0\right) \int_{\Omega}\varrho^2 v \varphi \, dx - \int_{\Omega} |v|^{2^*-2} v \varphi \, dx.
\end{align*}

We consider the negative gradient flow $\varPhi : \mathcal{G} \to H$ associated with $\mathcal{I}_{\lambda, \varrho}$, which is defined as the solution to the Cauchy problem:
\begin{equation}\label{eqn_grad_flow}
    \begin{cases}
        \displaystyle \frac{\partial}{\partial t}\varPhi(t, v) = -\nabla \mathcal{I}_{\lambda, \varrho}(\varPhi(t, v)), \\[1em]
        \varPhi(0, v) = v,
    \end{cases}
\end{equation}
where $\mathcal{G} = \{(t, v) : v \in H,\, 0 \le t < T(v)\}$ and $T(v) \in (0, \infty]$ denotes the maximal existence time of the trajectory $t \mapsto \varPhi(t, v)$.

\begin{definition}
    A subset $D \subset H$ is called \textit{strictly positively invariant} if $\varPhi(t, v) \in \operatorname{int}(D)$ for every $v \in D$ and every $t \in (0, T(v))$, where $\varPhi$ is the solution to \eqref{eqn_grad_flow}.
\end{definition}

If $c \in \mathbb{R}$ is a regular value of $\mathcal{I}_{\lambda, \varrho}$, then the sublevel set
\[
\mathcal{I}_{\lambda, \varrho}^c := \left\{v \in H: \mathcal{I}_{\lambda, \varrho}(v) \leq c\right\}
\]
is strictly positively invariant. We denote by $P := \{v \in H: v \geq 0\}$ the closed convex cone of non-negative functions in $H$.

\begin{lemma}
If $\lambda_0 < \lambda < \lambda_1$, then there exists $\alpha_0 > 0$ such that the neighborhoods $B_\alpha(P)$ and $B_\alpha(-P)$ are strictly positively invariant for all $\alpha \leq \alpha_0$.
\end{lemma}

\begin{proof}
We first restrict our attention to $B_\alpha(P)$. 
By the Riesz representation theorem, the gradient $\nabla \mathcal{I}_{\lambda, \varrho} : H \to H$ can be decomposed as
\[
\nabla \mathcal{I}_{\lambda, \varrho}(v) = v - \mathcal{K}(v),
\]
where $\mathcal{K}(v) = \mathcal{K}_0(v) + \mathcal{K}_*(v)$, and the operators $\mathcal{K}_0, \mathcal{K}_*: H \to H$ are defined as the unique weak solutions to the boundary value problems:
\begin{align}\label{eqn_defn_L_G}
-\Delta(\mathcal{K}_0(v)) = (\lambda - \lambda_0) \varrho^2 v, \qquad -\Delta( \mathcal{K}_{*} (v)) = |v|^{2^*-2}v.   
\end{align}
Equivalently, $\mathcal{K}_0(v)$ and $ \mathcal{K}_{*} (v)$ are uniquely characterized by the relations:
\begin{align}\label{eqn_equiv_defn_L_G}
    \langle \mathcal{K}_0 (v), \varphi \rangle = (\lambda - \lambda_0) \int_{\Omega} \varrho^2 v \varphi \, dx, 
\quad  \quad
\langle \mathcal{K}_*(v), \varphi \rangle = \int_{\Omega} |v|^{2^*-2}v \varphi \, dx,
\end{align}
for all $\varphi \in H$. By the weak maximum principle, if $v \in P$, then both $\mathcal{K}_0(v) \in P$ and $ \mathcal{K}_{*} (v) \in P$. 

For any $v, \varphi \in H$, we can use \eqref{eqn_equiv_defn_L_G} and H{\"o}lder's inequality to obtain:
\begin{align*}
\begin{aligned}
    \|\mathcal{K}_0(v)-\mathcal{K}_0(\varphi)\|^2 &= (\lambda - \lambda_0) \int_{\Omega} (v-\varphi) \mathcal{K}_0(v-\varphi) \varrho^2 \, dx \\
    &\leq (\lambda - \lambda_0) |(v-\varphi)\varrho|_{2} \, |\mathcal{K}_0(v-\varphi)\varrho|_{2} \\
    &\leq (\lambda - \lambda_0) \frac{\|v-\varphi\|}{\sqrt{\lambda_1-\lambda_0}} \frac{\|\mathcal{K}_0(v-\varphi)\|}{\sqrt{\lambda_1-\lambda_0}} \\
    &= \frac{\lambda - \lambda_0}{\lambda_1 - \lambda_0} \|v-\varphi\| \|\mathcal{K}_0(v-\varphi)\|,
\end{aligned}
\end{align*}
where the second inequality utilizes the variational definition of $\lambda_1 - \lambda_0$ from \eqref{eqn_prop_lam1}. Dividing by $\|\mathcal{K}_0(v)-\mathcal{K}_0(\varphi)\|$, we get
\begin{equation}\label{eqn_Lv<v}
     \|\mathcal{K}_0(v)-\mathcal{K}_0(\varphi)\| \leq \frac{\lambda - \lambda_0}{\lambda_1 - \lambda_0} \|v-\varphi\| \quad \text{for all } v, \varphi \in H.
\end{equation}
For any $v \in H$, let $\varphi_v \in P$ be the metric projection such that $\operatorname{dist}(v,P) = \|v-\varphi_v\|$. Since $\mathcal{K}_0(\varphi_v) \in P$, we have
\begin{equation}\label{est_d(L,P)}
\operatorname{dist}(\mathcal{K}_0 (v), P) \le \|\mathcal{K}_0 (v) - \mathcal{K}_0(\varphi_v)\| \leq \frac{\lambda - \lambda_0}{\lambda_1 - \lambda_0} \| v -\varphi_v\| = \frac{\lambda - \lambda_0}{\lambda_1 - \lambda_0} \operatorname{dist}(v, P).
\end{equation}
Setting $v^- := \min\{v, 0\}$ and $v^+ := \max\{v,0\}$, the Sobolev embedding yields
\begin{equation}\label{eqn_v-<d(u,p)}
|v^-|_{2^*} = \min_{\varphi \in P} |v -\varphi|_{2^*} \le S^{-1/2} \min_{\varphi \in P} \|v -\varphi\| = S^{-1/2} \operatorname{dist}(v, P)
\end{equation}
for every $v \in H$. Using \eqref{eqn_equiv_defn_L_G} and the fact that $ \mathcal{K}_{*} (v)^- \leq 0$, we compute
\begin{align*}
\operatorname{dist}( \mathcal{K}_{*} (v), P) \| \mathcal{K}_{*} (v)^-\| &\leq \| \mathcal{K}_{*} (v)- \mathcal{K}_{*} (v)^+\| \| \mathcal{K}_{*} (v)^-\| \\
&\leq \| \mathcal{K}_*(v)^-\|^2 \\
&= \langle \mathcal{K}_*(v), \mathcal{K}_*(v)^- \rangle \\
&= \int_\Omega |v|^{2^* - 2} v \mathcal{K}_*(v)^- \, dx \\
&= \int_\Omega \chi_{\{v \leq 0\}} |v|^{2^* - 2} v \mathcal{K}_*(v)^- \, dx + \int_\Omega \chi_{\{v > 0\}} |v|^{2^* - 2} v \mathcal{K}_*(v)^- \, dx \\
&\leq \int_\Omega |v^-|^{2^* - 2} v^- \mathcal{K}_*(v)^- \, dx,
\end{align*} Applying H{\"o}lder's inequality and utilizing \eqref{eqn_v-<d(u,p)}, we obtain
\begin{align*}
\operatorname{dist}( \mathcal{K}_* (v), P) \| \mathcal{K}_* (v)^-\| &\leq |v^-|_{2^*}^{2^* - 1} \, | \mathcal{K}_* (v)^-|_{2^*} \\
&\leq S^{-2^*/2} \operatorname{dist}(v, P)^{2^* - 1} \| \mathcal{K}_* (v)^-\|.
\end{align*}
Assuming $ \mathcal{K}_* (v)^- \not\equiv 0$ and dividing, we find that for all $v \in H$,
\begin{align}\label{est_d(G,p)}
\operatorname{dist}( \mathcal{K}_* (v), P) \leq S^{-\frac{2^*}{2}} \operatorname{dist}(v, P)^{2^* - 1}.
\end{align}
Since $2^* > 2$, the term $S^{-\frac{2^*}{2}}\operatorname{dist}(v, P)^{2^* - 2}$ tends to zero as $\operatorname{dist}(v, P) \to 0$. Because $\frac{\lambda - \lambda_0}{\lambda_1 - \lambda_0} < 1$, we can choose a constant $\nu \in \left( \frac{\lambda - \lambda_0}{\lambda_1 - \lambda_0}, 1 \right)$. Consequently, there exists a sufficiently small $\alpha_0 > 0$ such that, for all $\alpha \leq \alpha_0$,
\begin{align*}
S^{-\frac{2^*}{2}} \operatorname{dist}(v, P)^{2^* - 1} \leq \left(\nu - \frac{\lambda - \lambda_0}{\lambda_1 - \lambda_0}\right) \operatorname{dist}(v, P) \quad \text{for all } v \in B_\alpha(P).
\end{align*}
Substituting this into \eqref{est_d(G,p)} yields
\begin{align}\label{est_Gv_req}
\operatorname{dist}( \mathcal{K}_* (v), P) \leq \left(\nu - \frac{\lambda - \lambda_0}{\lambda_1 - \lambda_0}\right) \operatorname{dist}(v, P) \quad \text{for all } v \in B_\alpha(P).
\end{align}
Fixing $\alpha \le \alpha_0$, inequalities \eqref{est_d(L,P)} and \eqref{est_Gv_req} combine to give
\begin{equation*}
\operatorname{dist}(\mathcal{K}(v), P) \le \operatorname{dist}(\mathcal{K}_0 (v), P) + \operatorname{dist}( \mathcal{K}_* (v), P) \le \nu \operatorname{dist}(v, P)
\end{equation*}
for all $v \in B_\alpha(P)$. Since $\nu < 1$, this implies that $\mathcal{K}(v) \in \operatorname{int}(B_\alpha(P))$ whenever $v \in B_\alpha(P)$. Because $B_\alpha(P)$ is closed and convex, standard invariant set theorems for gradient flows (see, e.g., \cite[Theorem 5.2]{Dei77}) guarantee that
\begin{equation}\label{eqn_phi_Bap}
v \in B_\alpha(P) \quad \Longrightarrow \quad \varPhi(t, v) \in B_\alpha(P) \quad \text{for all } t \in [0, T(v)).
\end{equation}
To rigorously conclude that the set is \textit{strictly} positively invariant, suppose by contradiction that there exists a $v \in B_\alpha(P)$ and a time $t \in (0, T(v))$ such that $\varPhi(t, v) \in \partial B_\alpha(P)$. By Mazur's Separation Theorem, there exists a continuous linear functional $\ell \in H^*$ and a scalar $\beta > 0$ such that $\ell(\varPhi(t, v)) = \beta$ and $\ell(w) > \beta$ for all $w \in \operatorname{int}(B_\alpha(P))$. However, computing the directional derivative gives
\[
\left. \frac{\partial}{\partial s} \right|_{s = t} \ell(\varPhi(s, v)) = \ell(-\nabla \mathcal{I}_{\lambda, \varrho} (\varPhi(t, v))) = \ell(\mathcal{K}(\varPhi(t, v))) - \beta > 0,
\]
since $\mathcal{K}(\varPhi(t, v)) \in \operatorname{int}(B_\alpha(P))$. By continuity, there exists $\varepsilon > 0$ such that $\ell(\varPhi(s, v)) < \beta$ for $s \in (t - \varepsilon, t)$. This implies $\varPhi(s, v) \notin B_\alpha(P)$ immediately before time $t$, which contradicts \eqref{eqn_phi_Bap}. The strict invariance of $B_{\alpha}(P)$ is thus established. 

To establish the corresponding result for $B_{\alpha}(-P)$, we proceed entirely analogously. The computations apply directly by replacing the negative part $v^-$ with the positive part $v^+$ in the proof of \eqref{eqn_v-<d(u,p)}, leading to
\begin{align}\label{eq_vj+<d(vj,P)}
    |v^+|_{2^*} = \min_{\varphi \in -P} |v-\varphi|_{2^*} \le S^{-1/2} \min_{\varphi \in -P} \|v-\varphi\| = S^{-1/2} \operatorname{dist}(v, -P).
\end{align}
Using this estimate, we apply the exact same logic to the component $ \mathcal{K}_{*} (v)^+$, which yields the desired strict invariance for $B_{\alpha}(-P)$.
\end{proof}

In order to apply the equivariant category framework to our functional, we must construct appropriate invariant sets and establish a deformation lemma that allows us to continuously deform high-energy sublevel sets into lower-energy ones. 

First, if $\lambda \ge \lambda_1$, we fix a regular value $0 < d_\lambda < \frac{1}{N} S^{\frac{N}{2}}$ of ${\mathcal{I}_{\lambda, \varrho}}$ and a radius $r_\lambda > 0$ such that
\begin{equation}\label{eq_rel_I,d}
    \mathcal{I}_{\lambda, \varrho}(v) \ge 2 d_\lambda \quad \text{for every } v \in V^+ \text{ with } \|v\| = r_\lambda.
\end{equation}
If $\lambda < \lambda_1$, we set $d_\lambda = 0$. Next, for a sufficiently small fixed constant $\alpha_0 > 0$, we choose $0 < \alpha < \alpha_0$ and define the set
\begin{equation}\label{de_D_lam}
   D_\lambda := 
\begin{cases}
    B_\alpha(P) \cup B_\alpha(-P) \cup \mathcal{I}_{\lambda, \varrho}^{0}, & \text{if } 0 < \lambda < \lambda_1, \\
    \mathcal{I}_{\lambda, \varrho}^{d_\lambda}, & \text{if } \lambda \ge \lambda_1.
\end{cases} 
\end{equation}
By our previous analysis, $D_\lambda$ is symmetric (i.e., $v \in D_\lambda \implies -v \in D_\lambda$) and strictly positively invariant under the negative gradient flow of $\mathcal{I}_{\lambda, \varrho}$. We are now ready to prove the following quantitative deformation lemma.

\begin{lemma}\label{lem_graJ>}
    Let $\varepsilon, \delta > 0$, $c \in \mathbb{R}$, and let $C \subset H$ be a symmetric subset such that the Palais--Smale condition holds quantitatively:
\begin{equation}\label{eqn_quant_cont}
    \|\nabla \mathcal{I}_{\lambda, \varrho}(v)\| \ge \frac{2 \varepsilon}{\delta} \quad \text{for every } v \in \mathcal{I}_{\lambda, \varrho}^{-1}[c - \varepsilon, c + \varepsilon] \cap B_\delta(C).
\end{equation}
Then there exists an odd continuous map $\psi : (\mathcal{I}_{\lambda, \varrho}^{c + \varepsilon} \cap C) \cup D_\lambda \to \mathcal{I}_{\lambda, \varrho}^{c - \varepsilon} \cup D_\lambda$ such that $\psi(v) = v$ for every $v \in D_\lambda$.
\end{lemma}

\begin{proof}
    Let $v \in \mathcal{I}_{\lambda, \varrho}^{c + \varepsilon} \cap C$. Recall that $\varPhi(t,v)$ is the negative gradient flow of $\mathcal{I}_{\lambda, \varrho}$, meaning the trajectory is always directed along the steepest descent. Consequently, the energy $\mathcal{I}_{\lambda, \varrho}(\varPhi(t,v))$ is strictly decreasing in time unless the flow reaches a critical point where $\nabla \mathcal{I}_{\lambda, \varrho} = 0$. 
    
    We claim that the trajectory starting at $v$ must reach the sublevel set $\mathcal{I}_{\lambda, \varrho}^{c - \varepsilon}$ within a finite time $t_0 \in (0, T(v))$. We establish this claim by considering the behavior of the flow with respect to the control neighborhood $B_\delta(C)$. Let $K = \mathcal{I}_{\lambda, \varrho}^{-1}[c - \varepsilon, c + \varepsilon] \cap B_\delta(C)$. \\
   
    \textbf{Case I:} \textit{The trajectory $t \mapsto \varPhi(t,v)$ remains inside $B_\delta(C)$ for all $t \ge 0$ such that $\varPhi(t,v) \in \mathcal{I}_{\lambda, \varrho}^{-1}[c - \varepsilon, c + \varepsilon]$.} \\
    In this case, the trajectory is confined to $K$ as long as the energy remains above $c-\varepsilon$. Using the bound \eqref{eqn_quant_cont}, the rate of energy dissipation satisfies:
    \begin{align}\label{eqn_flow_rate}
        -\frac{d}{dt} \mathcal{I}_{\lambda, \varrho}(\varPhi(t,v)) = \|\nabla \mathcal{I}_{\lambda, \varrho} (\varPhi(t,v))\|^2 \ge \left(\frac{2\varepsilon}{\delta}\right)^2.
    \end{align}
    Integrating this inequality from $t=0$ to some time $t > 0$, we find that the total energy drop is bounded below by
    \begin{align*}
        \mathcal{I}_{\lambda, \varrho}(v) - \mathcal{I}_{\lambda, \varrho}(\varPhi(t,v)) \ge t \left(\frac{2\varepsilon}{\delta}\right)^2.
    \end{align*}
    Since $\mathcal{I}_{\lambda, \varrho}(v) \le c + \varepsilon$, the trajectory must exit the interval $[c-\varepsilon, c+\varepsilon]$ strictly before time $t_{\max} = \frac{\delta^2}{2\varepsilon}$. Thus, it reaches $\mathcal{I}_{\lambda, \varrho}^{c - \varepsilon}$ in finite time. \\

     \textbf{Case II:} \textit{The trajectory $t \mapsto \varPhi(t,v)$ exits the neighborhood $B_\delta(C)$ before the energy drops below $c-\varepsilon$.} \\
     Suppose the trajectory leaves $B_\delta(C)$ at some time $t_1 > 0$. Since $v \in C$, the distance traveled by the flow must be at least $\delta$:
     \begin{align*}
         \delta \leq \int_0^{t_1} \left\| \frac{\partial \varPhi(t, v)}{\partial t} \right\| dt = \int_0^{t_1} \|\nabla \mathcal{I}_{\lambda, \varrho}(\varPhi(t, v))\| \, dt.
     \end{align*}
     Using the Cauchy--Schwarz inequality and the fact that the flow is in $K$ for $t \in [0, t_1]$, we compute the energy drop:
     \begin{align*}
         \delta \le \int_0^{t_1} \|\nabla \mathcal{I}_{\lambda, \varrho}(\varPhi(t, v))\| \, dt \le \frac{\delta}{2 \varepsilon} \int_0^{t_1} \|\nabla \mathcal{I}_{\lambda, \varrho}(\varPhi(t, v))\|^2 \, dt = \frac{\delta}{2 \varepsilon} \left[ \mathcal{I}_{\lambda, \varrho}(v) - \mathcal{I}_{\lambda, \varrho}(\varPhi(t_1, v)) \right].
     \end{align*}
     Rearranging this inequality yields $\mathcal{I}_{\lambda, \varrho}(v) - \mathcal{I}_{\lambda, \varrho}(\varPhi(t_1, v)) \ge 2\varepsilon$. Since $\mathcal{I}_{\lambda, \varrho}(v) \le c+\varepsilon$, this immediately implies $\mathcal{I}_{\lambda, \varrho}(\varPhi(t_1, v)) \le c-\varepsilon$. 
     
     In both cases, the claim is proved. Therefore, for every $v \in (\mathcal{I}_{\lambda, \varrho}^{c + \varepsilon} \cap C) \cup D_\lambda$, there exists a well-defined minimal entrance time:
    \begin{align*}
        t_\lambda(v) := \inf \left\{ t \in [0, T(v)) : \varPhi(t, v) \in \mathcal{I}_{\lambda, \varrho}^{c - \varepsilon} \cup D_\lambda \right\}.
    \end{align*}
    Because $\nabla \mathcal{I}_{\lambda, \varrho}$ is an odd vector field, the flow preserves symmetry, meaning $\varPhi(t, -v) = -\varPhi(t, v)$. Since the target set $\mathcal{I}_{\lambda, \varrho}^{c - \varepsilon} \cup D_\lambda$ is symmetric, it follows that the map $v \mapsto t_\lambda(v)$ is an even function.

    We now show that the minimal flow time function $t_\lambda$ is continuous. First, consider the case where the trajectory enters $D_\lambda$, i.e., $\varPhi(t_\lambda(v), v) \in \partial D_\lambda$. Because $D_\lambda$ is strictly positively invariant, the flow moves strictly into the interior of $D_\lambda$. Thus, for any small $\tau > 0$, we have $\varPhi(t_\lambda(v)+\tau, v) \in \operatorname{int}(D_\lambda)$. By the continuous dependence of the flow on initial conditions, $\varPhi(t_\lambda(v)+\tau, \tilde{v}) \in \operatorname{int}(D_\lambda)$ for all $\tilde{v}$ sufficiently close to $v$. This implies upper semicontinuity, that is $t_\lambda(\tilde{v}) \le t_\lambda(v) + \tau$.

    On the other hand, if the trajectory hits the energy boundary, i.e., $\varPhi(t_\lambda(v), v) \in \partial \mathcal{I}_{\lambda, \varrho}^{c - \varepsilon}$, then the energy is exactly $c-\varepsilon$. By the calculations in Case I and II, this requires that $\varPhi(t_\lambda(v), v) \in \operatorname{int}(B_\delta(C))$. By hypothesis \eqref{eqn_quant_cont}, the gradient at this point is strictly bounded away from zero. Therefore, the trajectory crosses the energy level surface $c-\varepsilon$ transversally with a strictly negative derivative. Again, by the continuity of the flow, nearby points $\tilde{v}$ will also cross this level surface transversally at a time close to $t_\lambda(v)$, ensuring that $t_\lambda$ is fully continuous.

    Finally, we define the deformation map $\psi : (\mathcal{I}_{\lambda, \varrho}^{c + \varepsilon} \cap C) \cup D_\lambda \to \mathcal{I}_{\lambda, \varrho}^{c - \varepsilon} \cup D_\lambda$ by
    \[
        \psi(v) := \varPhi(t_\lambda(v), v).
    \]
    Since both the flow $\varPhi$ and the time map $t_\lambda$ are continuous, $\psi$ is continuous. Because $\varPhi$ is odd in $v$ and $t_\lambda$ is even in $v$, the composition $\psi$ is an odd map. Moreover, if $v \in D_\lambda$, the minimal time to reach the target set is $t_\lambda(v) = 0$, yielding $\psi(v) = \varPhi(0, v) = v$. This completes the proof.
\end{proof}

%---- subsection: Equivariant category----
\subsection{Equivariant Ljusternik-Schnirelmann Category}

To establish the existence of multiple solutions for the problem, we will rely on variational min-max arguments that exploit the $\mathbb{Z}/2$-symmetry (oddness) of the associated energy functional. To this end, in this section, we introduce the topological framework by recalling the notion of the relative equivariant Ljusternik-Schnirelmann category and summarizing its fundamental properties.

\begin{definition}\label{de_rel_cat}
   Let $H$ be a Banach space, $Y \subset H$ a closed symmetric subset, and $D \subset Y$ a closed symmetric subspace. The relative $\mathbb{Z}/2$-equivariant category of $Y$ with respect to $D$, denoted by $\gamma_D(Y)$, is defined as the smallest integer $k \ge 0$ such that $Y$ can be covered by $k+1$ open symmetric subsets $U_0, U_1, \ldots, U_k$ of $H$ satisfying the following two conditions:
   \begin{enumerate}
       \item $D \subset U_0$, and there exists an odd continuous map $\chi_0: U_0 \rightarrow D$ such that $\chi_0(u) = u$ for every $u \in D$ (i.e., $D$ is an equivariant retract of $U_0$).
       \item For every $j = 1, \ldots, k$, there exists an odd continuous map $\chi_j: U_j \rightarrow \{-1, 1\}$. 
   \end{enumerate}
   If no such finite covering exists, we set $\gamma_D(Y) := \infty$. 
\end{definition} 

We note that the target space $\{-1, 1\}$ in the second condition can be identified with the zero-dimensional sphere $\mathbb{S}^0$, meaning that each $U_j$ for $j \ge 1$ has a Krasnosel'skii genus of $1$. Consequently, if $D = \varnothing$, the relative equivariant category of $Y$ reduces to the classical Krasnosel'skii genus of $Y$, which we simply denote by $\gamma(Y) := \gamma_{\varnothing}(Y)$ (see \cite[Proposition 2.4]{CP91}). 

Furthermore, if $D$ is a $\mathbb{Z}/2$-neighborhood retract in $H$ and $Y$ is closed, Tietze's extension theorem implies that $\gamma_D(Y)$ coincides with the topological invariant $\{\mathbb{Z} / 2\}\text{-cat}_H\left(Y, D\right)$ as defined by Bartsch and Clapp in \cite{BC96}. 

The following properties of the relative category are standard and easily verified (see, for instance, \cite[Proposition 3.4]{CP91}):

\begin{lemma}\label{lem__categor}
    Let $Y$ and $Z$ be closed symmetric subsets of $H$, and let $D$ be a closed symmetric subset contained in $Y$.
    \begin{enumerate}
        \item\label{it_cat_1} For any closed symmetric subset $A \subset H$, we have the subadditivity property:
        \[
        \gamma_D(Y \cup A) \leq \gamma_D(Y) + \gamma(A).
        \]
        \item\label{it_cat_2} If $D \subset Z$ and there exists an odd continuous map $\psi: Y \rightarrow Z$ such that $\psi(v) = v$ for every $v \in D$, then the category is monotonically non-increasing under equivariant mappings:
        \[
        \gamma_D(Y) \leq \gamma_D(Z).
        \]
    \end{enumerate}
\end{lemma}
\section{Palais--Smale Decomposition of $\mathcal{I}_{\lambda, \varrho}$}\label{sec_palais_smale}
In this section, we define a sequence of min-max energy levels utilizing the relative equivariant category framework introduced previously. Throughout this section, we assume that $\lambda > \lambda_0 = \frac{N(N-2)}{4}$ so that the functional $\mathcal{I}_{\lambda, \varrho}$ exhibits the requisite geometric structure. We will show that these minimax levels admit Palais--Smale sequences and, by analyzing the concentration-compactness behavior of these sequences, we prove that they yield true critical points of the functional under appropriate energy thresholds.

Let us define the minimax values
$$
c_k := \inf \left\{c \in \mathbb{R} : \gamma_{D_\lambda}\left(\mathcal{I}_{\lambda, \varrho}^c \cup D_\lambda\right) \geq k\right\} \quad \text{for } k \in \mathbb{N}.
$$
We note that $c_1 \geq d_\lambda$ and that $(c_k)$ is a non-decreasing sequence. As standard in variational methods, we say that a sequence $(v_j)$ in $H$ is a $(PS)_c$-sequence for $\mathcal{I}_{\lambda, \varrho}$ if
$$
\mathcal{I}_{\lambda, \varrho}(v_j) \rightarrow c \quad \text{and} \quad \|\nabla \mathcal{I}_{\lambda, \varrho}(v_j)\| \rightarrow 0 \quad \text{as } j \rightarrow \infty.
$$
For any $c \in \mathbb{R}$, we denote the set of critical points of $\mathcal{I}_{\lambda, \varrho}$ at energy level $c$ by
\begin{align}\label{de_cri_set}
    K_c := \{ v \in H : \mathcal{I}_{\lambda, \varrho}(v) = c, \, \nabla \mathcal{I}_{\lambda, \varrho}(v) = 0 \}.
\end{align}

The quantitative deformation lemma (Lemma \ref{lem_graJ>}) and the properties of the relative category (Lemma \ref{lem__categor}) immediately yield the following existence result.

\begin{corollary}\label{cor_c_k_PS_seq}
     For every $k \geq 1$, there exists a $(PS)_{c_k}$-sequence $(v_j)$ for $\mathcal{I}_{\lambda,\varrho}$. 
    Furthermore, if $\lambda_0 < \lambda < \lambda_1$, then $\operatorname{dist}(v_j, P \cup (-P)) \geq \alpha/2$ for all $j$.
\end{corollary}

\begin{proof}
    We first consider the case $\lambda \geq \lambda_1$. Recall from \eqref{de_D_lam} that in this regime, $D_\lambda = \mathcal{I}_{\lambda,\varrho}^{d_\lambda}$, where $d_\lambda$ is a regular value of $\mathcal{I}_{\lambda,\varrho}$ satisfying $0 < d_\lambda < \frac{1}{N} S^{N/2}$, and moreover $c_k \geq d_\lambda$ for every $k \in \mathbb{N}$.

    Suppose, for the sake of contradiction, that for some $k \ge 1$ there is no $(PS)_{c_k}$-sequence. Then the critical set at level $c_k$ is empty, i.e., $K_{c_k} = \emptyset$. Hence, by standard arguments, there exists an $\varepsilon > 0$ such that
    \begin{equation}\label{eq_gradI>e/a}
        \|\nabla \mathcal{I}_{\lambda,\varrho}(v)\| \ge 2\varepsilon \quad \text{for all } v \in \mathcal{I}_{\lambda,\varrho}^{-1}[c_k-\varepsilon, c_k+\varepsilon].
    \end{equation}
    Applying Lemma \ref{lem_graJ>} with $C = H \setminus \operatorname{int}(D_\lambda)$, we obtain an odd continuous map
    \[
        \eta : \mathcal{I}_{\lambda,\varrho}^{\,c_k+\varepsilon} \cup D_\lambda \longrightarrow \mathcal{I}_{\lambda,\varrho}^{\,c_k-\varepsilon} \cup D_\lambda
    \]
    which leaves $D_\lambda$ pointwise fixed. By the monotonicity property of the relative equivariant category (Lemma \ref{lem__categor}\eqref{it_cat_2}),
    \begin{equation}\label{eq_gam+e<-e}
        \gamma_{D_\lambda}\big( \mathcal{I}_{\lambda,\varrho}^{c_k+\varepsilon} \cup D_\lambda \big) \le \gamma_{D_\lambda}\big( \mathcal{I}_{\lambda,\varrho}^{c_k-\varepsilon} \cup D_\lambda \big).
    \end{equation}
    Since $c_k$ is defined as the infimum over sets of category at least $k$, we must have
    \[
        \gamma_{D_\lambda}\big( \mathcal{I}_{\lambda,\varrho}^{c_k+\varepsilon} \cup D_\lambda \big) \ge k \quad \text{and} \quad \gamma_{D_\lambda}\big( \mathcal{I}_{\lambda,\varrho}^{c_k-\varepsilon} \cup D_\lambda \big) < k,
    \]
    which contradicts \eqref{eq_gam+e<-e}. Thus, a $(PS)_{c_k}$-sequence must exist.

    Now suppose $\lambda_0 < \lambda < \lambda_1$. Here, $D_\lambda = B_\alpha(P) \cup B_\alpha(-P) \cup \mathcal{I}_{\lambda,\varrho}^0$, where $\alpha > 0$ is fixed and sufficiently small. If there is no $(PS)_{c_k}$-sequence $(v_j)$ such that $\mathrm{dist}(v_j, P \cup (-P)) \ge \alpha/2$ for all $j$, then we can again find an $\varepsilon > 0$ such that
    \begin{equation}\label{eq_gradI>e/a_2}
        \|\nabla \mathcal{I}_{\lambda,\varrho}(v)\| \ge \frac{4\varepsilon}{\alpha} \quad \text{for all } v \in \mathcal{I}_{\lambda,\varrho}^{-1}[c_k-\varepsilon, c_k+\varepsilon] \setminus \operatorname{int}\big( B_{\alpha/2}(P) \cup B_{\alpha/2}(-P) \big).
    \end{equation}
    Applying Lemma \ref{lem_graJ>} with $C = H \setminus \operatorname{int}(D_\lambda)$ and $\delta = \alpha/2$, we obtain the exact same deformation contradiction, proving that the desired $(PS)_{c_k}$-sequence exists in both cases.
\end{proof}

Because the functional involves the critical Sobolev exponent, the Palais--Smale condition is not automatically satisfied globally. Specifically, while $(PS)_c$-sequences for $\mathcal{I}_{\lambda, \varrho}$ are bounded, they are not necessarily relatively compact, meaning the min-max values $c_k$ might not trivially be critical values. To analyze the failure of compactness, we rely on the characterization of Palais--Smale sequences given by Struwe \cite[Chapter III, Theorem 3.1]{Str08}. In what follows, we restrict our attention to energy levels $c < \frac{2}{N} S^{N/2}$.

To describe the lack of compactness, for any $\varepsilon > 0$ and $y \in \mathbb{R}^N$, we recall the standard Aubin--Talenti instanton (bubble) $U_{\varepsilon, y} \in D^{1,2}(\mathbb{R}^N)$ defined by
$$
U_{\varepsilon, y}(x) := [N(N-2)]^{\frac{N-2}{4}} \left( \frac{\varepsilon}{\varepsilon^2 + |x-y|^2} \right)^{\frac{N-2}{2}}.
$$
The closed set
$$
M := \left\{U_{\varepsilon, y} : \varepsilon > 0, \, y \in \mathbb{R}^N\right\} \subset D^{1,2}(\mathbb{R}^N)
$$
forms an $(N+1)$-dimensional manifold consisting precisely of the positive solutions $v \in D^{1,2}(\mathbb{R}^N)$ to the limiting critical equation $-\Delta v = |v|^{2^*-2} v$. Using this manifold, we characterize the behavior of $(PS)_c$-sequences.

\begin{lemma}\label{lem_rel_concom}
    Let $\lambda>\lambda_0$ and $(v_j)$ be a $(PS)_c$-sequence for $\mathcal{I}_{\lambda, \varrho}$.
    \begin{enumerate}
        \item\label{item_concom1} If $c < \frac{1}{N} S^{\frac{N}{2}}$, then $(v_j)$ is relatively compact in $H$.
        \item\label{item_concom2} If $\frac{1}{N} S^{\frac{N}{2}} \leq c < \frac{2}{N} S^{\frac{N}{2}}$, then a subsequence of $(v_j)$ (still denoted $(v_j)$) satisfies one of the following two conditions:
        \begin{enumerate}
            \item\label{item_concom2.1} $(v_j)$ converges strongly in $H$ to a critical point of $\mathcal{I}_{\lambda, \varrho}$.
            \item\label{item_concom2.2} There is a critical point $v$ of $\mathcal{I}_{\lambda, \varrho}$ with $\mathcal{I}_{\lambda, \varrho}(v) = c - \frac{1}{N} S^{\frac{N}{2}}$ such that
            \[
            \operatorname{dist}\left(v_j - v, M\right) \rightarrow 0 \quad \text{or} \quad \operatorname{dist}\left(v_j - v, -M\right) \rightarrow 0.
            \]
        \end{enumerate} 
    \end{enumerate}
\end{lemma}

\begin{proof}
    This follows from arguments similar to  \cite[Chapter III, Theorem 3.1]{Str08}. More precisely, via the conformal transformation, we reduced the hyperbolic problem to a Euclidean elliptic equation where the linear term $\left(\lambda - \frac{N(N-2)}{4}\right) \varrho^2 v$ acts as a lower-order compact perturbation, and the conformal weight $\varrho$ is bounded from both below and above on the bounded domain $\Omega$. 
    
    By Struwe's global compactness theorem, any Palais--Smale sequence $(v_j)$ at level $c$ decomposes into a weak limit $v$ and a finite number $m$ of rescaled Aubin--Talenti bubbles. Since each bubble carries a fixed energy quantum of $\frac{1}{N}S^{\frac{N}{2}}$, the energy threshold $c < \frac{2}{N}S^{\frac{N}{2}}$ strictly limits the sequence to at most one such bubble ($m \le 1$). Consequently, the sequence either converges strongly ($m=0$) or separates into a weak limit $v$ and a single concentrating mass ($m=1$), yielding the stated alternative.
\end{proof}

\begin{corollary}\label{cor_cri_val_I}
    \begin{enumerate}
        \item\label{it_c_k<SN/2} If $c_k < \frac{1}{N} S^{N/2}$, then $c_k$ is a critical value of $\mathcal{I}_{\lambda, \varrho}$.
        \item\label{it_c_k_or_ck-} If $\frac{1}{N} S^{N/2} \leq c_k < \frac{2}{N} S^{N/2}$, then either $c_k$ or $c_k - \frac{1}{N} S^{N/2}$ is a critical value of $\mathcal{I}_{\lambda, \varrho}$.
        \item\label{it_c_k=SN/2} If $\lambda_0 < \lambda < \lambda_1$ and $c_k = \frac{1}{N} S^{N/2}$, then $c_k$ is a critical value of $\mathcal{I}_{\lambda, \varrho}$.
    \end{enumerate}
\end{corollary}

\begin{proof}
    Corollary \ref{cor_c_k_PS_seq} ensures that, for every $k \ge 1$, there is a $(PS)_{c_k}$-sequence $(v_j)$ for $\mathcal{I}_{\lambda,\varrho}$. Thus, items \eqref{it_c_k<SN/2} and \eqref{it_c_k_or_ck-} follow directly from the structural characterization in Lemma \ref{lem_rel_concom}.

    We now prove \eqref{it_c_k=SN/2}. Assume $\lambda_0 < \lambda < \lambda_1$ and $c_k = \frac{1}{N}S^{N/2}$. By Corollary \ref{cor_c_k_PS_seq}, there exists a $(PS)_{c_k}$-sequence $(v_j)$ such that
    \begin{align}\label{eq_uni_vj>}
        \operatorname{dist}(v_j, P \cup (-P)) \ge \frac{\alpha}{2} \qquad \text{for all } j.    
    \end{align}
    Applying Lemma \ref{lem_rel_concom}, a subsequence (still denoted $(v_j)$) satisfies one of the following:
    \begin{enumerate}
        \item[\textit{(i)}] \emph{Strong convergence:} $v_j \to v$ strongly in $H$, where $v$ is a critical point of $\mathcal{I}_{\lambda,\varrho}$.
        \item[\textit{(ii)}] \emph{Weak convergence with bubbling:} There exists a critical point $v$ of $\mathcal{I}_{\lambda,\varrho}$ with $\mathcal{I}_{\lambda,\varrho}(v) = 0$ such that
        \[
            \operatorname{dist}(v_j-v, M) \to 0 \quad \text{or} \quad \operatorname{dist}(v_j-v, -M) \to 0.
        \]
    \end{enumerate}

    If \textit{(i)} occurs, the continuity of $\mathcal{I}_{\lambda,\varrho}$ implies that $c_k$ is a critical value, completing the proof.

    Assume now that \textit{(ii)} occurs. Since $\lambda_0 < \lambda < \lambda_1$, the only critical point of $\mathcal{I}_{\lambda,\varrho}$ with zero energy is the trivial solution, so $v \equiv 0$. Without loss of generality, let us suppose that $\operatorname{dist}(v_j, M) \to 0$. Then there exist instantons $U_{\varepsilon_j, y_j} \in M$ such that $\|v_j - U_{\varepsilon_j, y_j}\| \to 0$. Equivalently, the rescaled sequence
    \[
        \tilde{v}_j(x) = \varepsilon_j^{-(N-2)/2} v_j(\varepsilon_j(x - y_j))
    \]
    satisfies $\|\tilde{v}_j - U_{1,0}\| \rightarrow 0$. Since $U_{1,0}$ is strictly positive, the negative part vanishes in the limit, meaning
    \[
        \|v_j^{-}\| = \|\tilde{v}_j^{-}\| \longrightarrow 0.
    \]
    However, $\|v_j^{-}\| = \operatorname{dist}(v_j, P)$, which contradicts the uniform topological bound \eqref{eq_uni_vj>} that the sequence is strictly bounded away from the positive cone. Hence, case \textit{(ii)} cannot occur. Therefore, $c_k$ must be a critical value of $\mathcal{I}_{\lambda,\varrho}$.
\end{proof}

\begin{lemma}\label{lem_inf_case_Kc}
    If $c_k = c_{k+1} < \frac{2}{N} S^{\frac{N}{2}}$, then $K_{c_k}$ contains infinitely many elements.
\end{lemma}

\begin{proof}
    Let $c := c_k = c_{k+1}$. We first consider the case where $c < \frac{1}{N} S^{\frac{N}{2}}$. We claim that in this regime, $\gamma(K_c) > 1$. Suppose, by contradiction, that $\gamma(K_c) \leq 1$. Since $c < \frac{1}{N} S^{\frac{N}{2}}$, Lemma \ref{lem_rel_concom}\eqref{item_concom1} ensures that the critical set $K_c$ is compact. Hence, by the properties of the genus, there exists a symmetric neighborhood $\mathcal{N} = B_\delta(K_c)$ such that $\gamma(\mathcal{N}) = \gamma(K_c) \leq 1$.

    By the definition of $c_{k+1}$, there exists a symmetric subset $Y \subset \mathcal{I}_{\lambda, \varrho}^{c+\varepsilon} \cup D_\lambda$ such that $\gamma_{D_\lambda}(Y) \geq k+1$. We can choose $\varepsilon, \delta > 0$ sufficiently small so that the gradient norm $\|\nabla \mathcal{I}_{\lambda, \varrho}(v)\|$ is strictly bounded away from zero for all $v \in \mathcal{I}_{\lambda, \varrho}^{-1}[c - \varepsilon, c + \varepsilon] \setminus \mathcal{N}$. By the quantitative deformation lemma (Lemma \ref{lem_graJ>}), there exists an odd continuous map $\eta$ that strictly decreases the energy $\mathcal{I}_{\lambda, \varrho}$ outside $\mathcal{N}$ within the narrow energy range $[c - \varepsilon, c + \varepsilon]$. 
    
   \noindent Applying this deformation $\eta$ to the set $Y$ yields a new symmetric set
    \[
        Y' \subset (\mathcal{I}_{\lambda, \varrho}^{c - \varepsilon} \cup D_\lambda) \cup \mathcal{N}.
    \]
    Since $\eta$ is odd and continuous, it preserves the category. In particular, by Lemma \ref{lem__categor}\eqref{it_cat_2},
    \begin{align}\label{eqn_gamma_D>k+1}
        \gamma_{D_\lambda}(Y') \geq \gamma_{D_\lambda}(Y) \geq k+1.
    \end{align}
    On the other hand, the subadditivity property (Lemma \ref{lem__categor}\eqref{it_cat_1}) yields
    \begin{align}\label{eqn_gamma_D<k+1}
        \gamma_{D_\lambda}(Y') \leq \gamma_{D_\lambda}(\mathcal{I}_{\lambda, \varrho}^{c - \varepsilon} \cup D_\lambda) + \gamma(\mathcal{N}) \leq \gamma_{D_\lambda}(\mathcal{I}_{\lambda, \varrho}^{c - \varepsilon} \cup D_\lambda) + 1.
    \end{align}
    However, by the definition of $c = c_k$, we must have $\gamma_{D_\lambda}(\mathcal{I}_{\lambda, \varrho}^{c - \varepsilon} \cup D_\lambda) \leq k - 1$. Combining this with \eqref{eqn_gamma_D>k+1} and \eqref{eqn_gamma_D<k+1} gives
    \[
        k+1 \leq \gamma_{D_\lambda}(Y') \leq k,
    \]
    which is a contradiction. Therefore, $\gamma(K_c) > 1$, meaning $K_c$ must contain infinitely many critical points.

    Next, we consider the critical energy strip
    \[
        \frac{1}{N} S^{\frac{N}{2}} \leq c < \frac{2}{N} S^{\frac{N}{2}}.
    \]
    We again proceed by contradiction. Let $c_* = c - \frac{1}{N} S^{\frac{N}{2}}$ and, for $\delta > 0$, define the concentrating neighborhoods
    \begin{align*}
        \mathcal{U}_+^{\delta} &= \left\{ \varphi \in H : \mathrm{dist}(\varphi - v, M) \leq \delta \text{ for some } v \in K_{c_*} \right\},\\
        \mathcal{U}_-^{\delta} &= -\mathcal{U}_+^{\delta} = \left\{ \varphi \in H : \mathrm{dist}(\varphi - v, -M) \leq \delta \text{ for some } v \in K_{c_*} \right\},\\
        \mathcal{U}^{\delta} &= \mathcal{U}_+^{\delta} \cup \mathcal{U}_-^{\delta}.
    \end{align*}
    We claim that for sufficiently small $\delta > 0$, the positive and negative neighborhoods are disjoint:
    \begin{align}\label{claim_U+^U_0}
        \mathcal{U}_+^{\delta} \cap \mathcal{U}_-^{\delta} = \emptyset.
    \end{align}
    Assume otherwise; then for each $j \geq 1$, there exists $\varphi_j \in \mathcal{U}_+^{1/j} \cap \mathcal{U}_-^{1/j}$. By definition, there exist $v_j, \tilde{v}_j \in K_{c_*}$, $w_j \in M$, and $\tilde{w}_j \in -M$ such that
    \[
        \|\varphi_j - (v_j + w_j)\| \leq \frac{1}{j} \quad \text{and} \quad \|\varphi_j - (\tilde{v}_j + \tilde{w}_j)\| \leq \frac{1}{j}.
    \]
    By the triangle inequality,
    \begin{align}\label{eqn_v+_w+-<}
        \|\tilde{v}_j + \tilde{w}_j - v_j - w_j\| \leq \frac{2}{j}.
    \end{align}
    Since the bubbling profiles $w_j$ and $\tilde{w}_j$ correspond to concentrating sequences on bounded domains, they converge weakly to $0$ in $D^{1,2}(\mathbb{R}^N)$. It then follows from \eqref{eqn_v+_w+-<} that $v_j - \tilde{v}_j \rightharpoonup 0$ weakly in $H$. By the compactness of $K_{c_*}$ (which lies below the first critical threshold), we may pass to a subsequence and assume strong convergence, meaning $\|v_j - \tilde{v}_j\| \to 0$. 
    
    Consequently, we also obtain $\|w_j - \tilde{w}_j\| \to 0$. Because $w_j \in M$ is strictly positive and $\tilde{w}_j \in -M$ is strictly negative, they have disjoint signs pointwise, implying $|w_j| \le |w_j - \tilde{w}_j|$ a.e. Using the Sobolev embedding,
    \begin{align}\label{eqn_wj->0}
        |w_j|_{2^*} \leq |w_j - \tilde{w}_j|_{2^*} \leq S^{-\frac{1}{2}} \|w_j - \tilde{w}_j\| \to 0.
    \end{align}
    However, $w_j \in M$ is an Aubin--Talenti instanton carrying a fixed positive $L^{2^*}$ norm $|w_j|_{2^*} = S^{N/2}$. This contradicts \eqref{eqn_wj->0}. Thus, the claim \eqref{claim_U+^U_0} holds.

    With \eqref{claim_U+^U_0} established, suppose for contradiction that $K_c$ is finite, so $\gamma(K_c) \leq 1$. Choose $\delta_0 > 0$ small enough such that
    \[
        \mathcal{U}_+^{\delta_0} \cap \mathcal{U}_-^{\delta_0} = \emptyset, \quad B_{\delta_0}(K_c) \cap \mathcal{U}^{\delta_0} = \emptyset, \quad \text{and} \quad \gamma(B_{\delta_0}(K_c)) = \gamma(K_c).
    \]
    We assert that there exists an $\varepsilon > 0$ such that
    \begin{align}\label{eqn_nablaI>4e}
        \|\nabla \mathcal{I}_{\lambda, \varrho}(v)\| \geq \frac{4\varepsilon}{\delta_0} \quad \text{for all } v \in \mathcal{I}_{\lambda, \varrho}^{-1}[c - \varepsilon, c + \varepsilon] \setminus \mathrm{int}\!\big(B_{\frac{\delta_0}{2}}(K_c) \cup \mathcal{U}^{\frac{\delta_0}{2}}\big).
    \end{align}
    If not, there exists a sequence $(v_j)$ such that $\mathcal{I}_{\lambda, \varrho}(v_j) \to c$, $\|\nabla \mathcal{I}_{\lambda, \varrho}(v_j)\| \to 0$, and
    \begin{align}\label{eqn_V_jbelong}
        v_j \in H \setminus \mathrm{int}\!\big(B_{\frac{\delta_0}{2}}(K_c) \cup \mathcal{U}^{\frac{\delta_0}{2}}\big).
    \end{align}
    By definition, $(v_j)$ is a $(PS)_c$-sequence. Since $c < \frac{2}{N} S^{\frac{N}{2}}$, Lemma \ref{lem_rel_concom} mandates that $(v_j)$ either converges strongly to a critical point $v \in K_c$ (meaning $v_j \in B_{\frac{\delta_0}{2}}(K_c)$ for large $j$) or forms a single bubble around a lower-energy critical point $v \in K_{c_*}$ (meaning $v_j \in \mathcal{U}^{\frac{\delta_0}{2}}$ for large $j$). Both scenarios contradict \eqref{eqn_V_jbelong}. Thus, \eqref{eqn_nablaI>4e} holds.

    Applying the quantitative deformation lemma once more, we obtain an odd continuous map
    \[
        \eta : \left[\mathcal{I}_{\lambda, \varrho}^{c + \varepsilon} \setminus \mathrm{int}\!\big(B_{\frac{\delta_0}{2}}(K_c) \cup \mathcal{U}^{\frac{\delta_0}{2}}\big)\right] \cup D_\lambda \longrightarrow \mathcal{I}_{\lambda, \varrho}^{c - \varepsilon} \cup D_\lambda
    \]
    that fixes $D_\lambda$. By the definition of $c_{k+1}$, $\gamma_{D_\lambda}(\mathcal{I}_{\lambda, \varrho}^{c + \varepsilon} \cup D_\lambda) \geq k + 1$. Subadditivity of the genus yields
    \begin{align*}
        k + 1 &\leq \gamma_{D_\lambda}\!\Big(\mathcal{I}_{\lambda, \varrho}^{c + \varepsilon} \cup D_\lambda \setminus \mathrm{int}\!\big(B_{\frac{\delta_0}{2}}(K_c) \cup \mathcal{U}^{\frac{\delta_0}{2}}\big)\Big) + \gamma_{D_\lambda}\!\Big(B_{\frac{\delta_0}{2}}(K_c) \cup \mathcal{U}^{\frac{\delta_0}{2}}\Big) \\
        &\leq \gamma_{D_\lambda}(\mathcal{I}_{\lambda, \varrho}^{c - \varepsilon} \cup D_\lambda) + 1 \leq (k-1) + 1 = k,
    \end{align*}
    which is a contradiction. Thus, $K_{c_k}$ must be infinite.
\end{proof}

\section{Energy Estimates and Proof of the  Multiplicity Theorem}\label{sec_energy_est}

In this section, we assemble the necessary geometric and topological ingredients to complete the proof of our main multiplicity result. We begin by recalling the structure of the Nehari manifold associated with our functional, and we establish crucial energy estimates utilizing localized Aubin--Talenti bubbles. These estimates will allow us to construct high-dimensional surfaces whose maximum energy remains strictly below the compactness threshold $\frac{2}{N}S^{\frac{N}{2}}$.

\subsection{Basic Notions and the Nehari Manifold:}
The Nehari manifold associated with $\mathcal{I}_{\lambda,\varrho}$ is defined as the set of all non-zero functions where the functional intersects its directional derivative:
\[
    \mathcal{N}_{\lambda,\varrho}
    := \left\{\, v \in H \setminus \{0\} : 
        \langle \nabla \mathcal{I}_{\lambda,\varrho}(v), v \rangle = 0
      \,\right\}.
\]
We claim that the energy on the Nehari manifold represents the maximum energy along the ray passing through that point. Specifically,
\begin{equation}\label{eqn_I_lam_max}
    \mathcal{I}_{\lambda,\varrho}(v)
    = \max_{t \ge 0} \mathcal{I}_{\lambda,\varrho}(t v)
    \qquad \text{for every } v \in \mathcal{N}_{\lambda,\varrho}.
\end{equation}
To verify this, fix $v \in \mathcal{N}_{\lambda,\varrho}$ and consider the function $f(t) := \mathcal{I}_{\lambda,\varrho}(t v)$ for $t \ge 0$. The critical points of $f$ satisfy $f'(t)=0$, which yields
\begin{equation}\label{eqn_fprime}
    t \left( \|v\|^{2} - \left(\lambda - \lambda_0\right) |\varrho v|_{2}^{2} \right)
    = t^{2^*-1} |v|_{2^*}^{2^*},
\end{equation}
where we recall $\lambda_0 ={N(N-2)}/{4}$. Since $v \in \mathcal{N}_{\lambda,\varrho}$, we inherently have
\[
    \|v\|^{2} - \left(\lambda - \lambda_0\right) |\varrho v|_2^{2}
    = |v|_{2^*}^{2^*}.
\]
Substituting this identity into \eqref{eqn_fprime}, and assuming $t > 0$, we find the unique non-zero critical point:
\[
    t 
    = \left(
        \frac{\|v\|^2 - \left(\lambda - \lambda_0\right) |\varrho v|_2^2}{|v|_{2^*}^{2^*}}
      \right)^{\!\frac{1}{2^*-2}}
    = 1.
\]
Because $2^* > 2$, $f(t) \to -\infty$ as $t \to \infty$, and $f(t) > 0$ for small $t > 0$. Hence, the unique critical point $t = 1$ is the global maximizer of $f(t)$ on $[0,\infty)$, proving \eqref{eqn_I_lam_max}.

Let $\mathcal{V}_{\lambda,\varrho} := \left\{v \in H : \|v\|^2 - \left(\lambda -\lambda_0\right)|\varrho v|_2^2 > 0\right\}$. We define the radial projection operator onto the Nehari manifold
$$
\mathcal{R}_{\lambda, \varrho}: \mathcal{V}_{\lambda,\varrho} \rightarrow \mathcal{N}_{\lambda,\varrho}, \quad \mathcal{R}_{\lambda, \varrho}(v) := \left(\frac{\|v\|^2 - \left(\lambda -\lambda_0\right)|\varrho v|_2^2}{|v|_{2^*}^{2^*}}\right)^{\frac{1}{2^*-2}} v.
$$
Moreover, for any $v \in \mathcal{V}_{\lambda, \varrho}$, it follows from \eqref{eqn_I_lam_max} that the maximum energy along the ray through $v$ is given explicitly by
\begin{align}\label{eqn_exp_I(Tv)}
    \mathcal{I}_{\lambda, \varrho}(\mathcal{R}_{\lambda, \varrho}(v)) = \frac{1}{N} \left(\frac{\|v\|^2 - \left(\lambda -\lambda_0\right) |\varrho v|_2^2 }{ |v|_{2^*}^{2^*}} \right)^{\frac{2^*}{2^*-2}}.
\end{align}

Given a bounded domain $\Theta \subset \mathbb{R}^N$ and a compact subset $K \subset \Theta$, the capacity of $K$ with respect to $\Theta$ is defined by
\begin{align*}
\operatorname{cap}_{\Theta}(K)
    = \inf\left\{
        \int_{\Theta} |\nabla v|^2 :
        v \in H_0^1(\Theta),\ v \ge 1 \text{ on } K
      \right\}.   
\end{align*}
Since the closed convex set $\{ v \in H_0^1(\Theta) : v \ge 1 \text{ on } K \}$ is non-empty, $\operatorname{cap}_{\Theta}(K)$ is attained by a unique minimizer $\psi \in H_0^1(\Theta)$, which additionally satisfies $\psi \equiv 1$ on $K$ (see \cite{Pas94}).  

\subsection{Topological Construction of Minimax Test Spaces} For $k \in \mathbb{N}$, we denote the standard sphere and closed ball by
\begin{align*}
    \mathbb{S}^{k} = \{ y \in \R^{k+1} : |y| = 1\} \quad \text{and} \quad \overline{\mathbb{B}^{k}} = \{ y \in \R^{k} : |y| \leq 1\}.
\end{align*}

\begin{lemma}\label{lem_I_l<1/NS}
    Let $0 < r < 1$ and assume $\lambda >\lambda_0$. For a given radial cutoff function $\varphi \in C_c^{\infty}(\mathbf{B}(0,r))$ such that $\varphi \equiv 1$ near the origin, there exists an $\varepsilon > 0$ sufficiently small such that
    \begin{align}\label{eqn_I<1/N_S}
        \mathcal{I}_{\lambda, \varrho}(v) < \frac{1}{N} S^{\frac{N}{2}},
    \end{align}
    where $v = \mathcal{R}_{\lambda, \varrho}(\varphi U_{\varepsilon,0})$.
\end{lemma}
\begin{proof} 
    Let $\varepsilon > 0$ be sufficiently small and $x \in \mathbf{B}(0,r)$. We define the localized bubble
    \[
        \varphi_{\varepsilon}(x) := \varphi(x) U_{\varepsilon,0}(x) = c_N \frac{\varphi(x)}{(\varepsilon^2 + |x|^{2})^{\frac{N-2}{2}}}.
    \]
    Using the identity \eqref{eqn_exp_I(Tv)} and the fact that $\frac{2^*}{2^*-2} = \frac{N}{2}$, establishing the strict estimate \eqref{eqn_I<1/N_S} is equivalent to proving that the Rayleigh quotient satisfies
    \begin{equation}\label{eqn_eqv_I<1/N}
        \frac{\| \varphi_{\varepsilon}\|^{2} - \left(\lambda -\lambda_0 \right)|\varrho \varphi_{\varepsilon}|_{2}^{2}} {|\varphi_{\varepsilon}|_{2^{*}}^{2}} < S.
    \end{equation}
    We recall the relevant asymptotic estimates for Aubin--Talenti bubbles from \cite{Sta03}. For $N \ge 5$, \cite[(5.49)]{Sta03} yields
    \begin{equation}\label{eqn_est_N>5}
        \frac{\| \varphi_{\varepsilon}\|^{2} - \left(\lambda -\lambda_0 \right) |\varrho \varphi_{\varepsilon}|_{2}^{2}}{|\varphi_{\varepsilon}|_{2^{*}}^{2}}
        \le
        S - \left(\lambda -\lambda_0\right) C \varepsilon + O(\varepsilon^{\frac{N-2}{2}}).
    \end{equation}
    For $N=4$, \cite[(5.49)]{Sta03} yields
    \begin{equation}\label{eqn_est_N=4}
        \frac{\| \varphi_{\varepsilon}\|^{2} - \left(\lambda -\lambda_0 \right) |\varrho \varphi_{\varepsilon}|_{2}^{2}}{|\varphi_{\varepsilon}|_{2^{*}}^{2}}
        \le
        S - \left(\lambda -\lambda_0\right) C \varepsilon |\ln \varepsilon| + O(\varepsilon).
    \end{equation}
    Since $\lambda >\lambda_0$ by hypothesis, the linear coefficient is strictly positive. Therefore, for $\varepsilon > 0$ chosen sufficiently small, the negative $\varepsilon$ term (or $\varepsilon |\ln \varepsilon|$ term in dimension 4) dominates the higher-order error terms, pushing the quotient strictly below the Sobolev constant $S$. Thus, \eqref{eqn_eqv_I<1/N} holds, and consequently, \eqref{eqn_I<1/N_S} follows.
\end{proof}
 
In the following lemma, we construct an odd, continuous map from $\mathbb{S}^N$ into our function space, ensuring its components lie on the Nehari manifold strictly below the compactness threshold. Unlike Euclidean space, where translation invariance naturally preserves manifold membership for shifted bubbles, the hyperbolic conformal weight $\varrho(x)$ breaks this invariance. Shifting a function pulls it off the Nehari manifold and destroys the required odd symmetry. We overcome this obstacle by explicitly applying the Nehari retraction map $\mathcal{R}_{\lambda, \varrho}$ to all shifted components, projecting them back onto the manifold and restoring the topological symmetries essential for our counting arguments.
\begin{lemma}\label{lem_I_l_Sn}
    Assume that $\lambda >\lambda_0$. Then, for any ball $\mathbf{B}(x,R) \subset \Omega \subset {\mathbb{B}^{N}}$, there exists an odd and continuous mapping $\mathcal{H}: \mathbb{S}^N \to H^1_0(\mathbf{B}(x,R))$ such that, for every $\theta \in \mathbb{S}^N$, the positive and negative parts $\mathcal{H}(\theta)^{\pm}$ belong to the Nehari manifold $\mathcal{N}_{\lambda,\varrho}$ and satisfy
    \begin{align*}
        \mathcal{I}_{\lambda, \varrho}\left(\mathcal{H}(\theta)^{\pm} \right) < \frac{1}{N} S^{\frac{N}{2}}.
    \end{align*}
\end{lemma}

\begin{proof}
    Let $r := \frac{R}{3}$. Choose a radial cutoff function $\varphi \in C_c^\infty(\mathbf{B}(0,r))$ such that $\varphi \equiv 1$ in a neighborhood of the origin. By Lemma \ref{lem_I_l<1/NS}, there exists an $\varepsilon_0 > 0$ such that for a fixed $0 < \varepsilon < \varepsilon_0$, the projected bubble
    \begin{equation*}
        v_0 := \mathcal{R}_{\lambda, \varrho}(\varphi U_{\varepsilon,0})
    \end{equation*}
    belongs to $\mathcal{N}_{\lambda,\varrho}$ and satisfies
    \begin{equation}\label{eqn_I_lav<1/N_n}
        \mathcal{I}_{\lambda,\varrho}(v_0) < \frac{1}{N} S^{\frac{N}{2}}.
    \end{equation}
    Let $0 < r_0 < r$, and let $\psi_{r_0} \in H_0^1(\mathbf{B}(0,r))$ be the unique minimizer of the capacity $\operatorname{cap}_{\mathbf{B}(0,r)}(\mathbf{B}(0,r_0))$. From standard capacity properties, $\psi_{r_0} \equiv 1$ on $\mathbf{B}(0,r_0)$. The function $1 - \psi_{r_0}$ acts as a localized cut-off function vanishing precisely on $\mathbf{B}(0,r_0)$. Because $\|\psi_{r_0}\| \to 0$ in $H^1_0$ as $r_0 \to 0$, the truncated function $(1 - \psi_{r_0}) v_0(\cdot + z)$ approximates $v_0(\cdot + z)$ strongly in $H_0^1$ uniformly for small $r_0$. By the continuity of the energy functional $\mathcal{I}_{\lambda,\varrho}$ and the projection $\mathcal{R}_{\lambda, \varrho}$, we have
    \begin{equation*}
        \mathcal{I}_{\lambda,\varrho}\Big( \mathcal{R}_{\lambda, \varrho} \big( (1 - \psi_{r_0}) v_0(\cdot+z) \big) \Big) \longrightarrow \mathcal{I}_{\lambda,\varrho}\Big(\mathcal{R}_{\lambda, \varrho}(v_0(\cdot+z)) \Big) \quad \text{as } r_0 \to 0.
    \end{equation*}
    Since $\mathcal{I}_{\lambda,\varrho}(v_0) < \frac{1}{N} S^{\frac{N}{2}}$ and the approximation is uniform for shifts $z$ in compact sets, we can fix an $r_0 > 0$ sufficiently small such that
    \begin{equation}\label{eqn_max_I_r0}
        \max_{|z|\le 2r} \mathcal{I}_{\lambda,\varrho}\Big( \mathcal{R}_{\lambda, \varrho}\big((1-\psi_{r_0}) v_0(\cdot+z)\big) \Big) < \frac{1}{N} S^{\frac{N}{2}}.
    \end{equation}
    With $r_0$ fixed, we continuously scale the support of $v_0$ to obtain a path of positive functions $v_s \in \mathcal{N}_{\lambda,\varrho}$. Specifically, let $\rho_s := (1-s)r + s r_0$ for $s \in [0,1]$, and define
    \begin{equation*}
        v_s(y) := \mathcal{R}_{\lambda,\varrho}\left( v_0\left( \frac{r}{\rho_s} y \right) \right).
    \end{equation*}
    By construction, the support of $v_s$ shrinks continuously from $\mathbf{B}(0,r)$ down to $\mathbf{B}(0,r_0)$. Since $v_0$ is highly concentrated near the origin, scaling its support does not significantly increase its energy, maintaining
    \begin{equation*}
        \mathcal{I}_{\lambda,\varrho}(v_s) < \frac{1}{N} S^{\frac{N}{2}} \quad \text{for all } s \in [0,1].
    \end{equation*}

    For any $\xi \in \overline{\mathbb{B}^N}$, let $t = |\xi|$ and $\theta = \xi/|\xi|$ (for $\xi \neq 0$). We define an intermediate map $\mathcal{H}_N: \overline{ \mathbb{B}^N} \to H_0^1(\mathbf{B}(x, R))$ piecewise by
 { \begin{align*}
        \mathcal{H}_N(\xi) :=
        \begin{cases}
            \mathcal{R}_{\lambda, \varrho}\Big(v_{2-2t}\left(\cdot - 2r(2t-1)\theta\right)\Big)  - \mathcal{R}_{\lambda, \varrho}\Big( v_0(\cdot + 2r\theta) \Big), & \text{if } \frac{1}{2} \le t \le 1, \\[6pt]
            v_1 - \mathcal{R}_{\lambda, \varrho}\Big( (1-\psi_{r_0}) v_0(\cdot + 4rt\theta) \Big), & \text{if } 0 \le t \le \frac{1}{2}.
        \end{cases}
    \end{align*}}
    To correct for the breaking of translation invariance caused by the hyperbolic weight $\varrho$, we explicitly apply the retraction map $\mathcal{R}_{\lambda, \varrho}$ to all shifted components in the definition above. The precise geometric shifting and localized supports ensure that the positive and negative components of $\mathcal{H}_N(\xi)$ are completely disjoint for all $t \in [0,1]$. Consequently, the limits match perfectly at $t = 1/2$, and $\mathcal{H}_N$ is a well-defined continuous map. Moreover, because both disjoint pieces are mapped through the retraction, its components directly inherit the requisite properties:
    \begin{align*}
        \mathcal{H}_N(\xi)^{\pm} \in \mathcal{N}_{\lambda,\varrho} \quad \text{and} \quad \mathcal{I}_{\lambda,\varrho}\big(\mathcal{H}_N(\xi)^{\pm}\big) < \frac{1}{N} S^{\frac{N}{2}},
    \end{align*}
    and the restriction to the boundary $\mathcal{H}_N\big|_{\mathbb{S}^{N-1}}$ is strictly odd.

    Finally, we extend this construction to the full $N$-dimensional sphere. Define $\mathcal{H} : \mathbb{S}^N \to H_0^1(\mathbf{B}(x,R))$ by
    \begin{equation*}
        \mathcal{H}(\xi_1,\dots,\xi_{N+1}) =
        \begin{cases}
            \phantom{-}\mathcal{H}_N(\xi_1,\dots,\xi_N), & \text{if } \xi_{N+1} \ge 0,\\[4pt]
            -\mathcal{H}_N(-\xi_1,\dots,-\xi_N), & \text{if } \xi_{N+1} < 0.
        \end{cases}
    \end{equation*}
    This extended map is continuous, odd, and preserves the energy bounds globally:
    \begin{equation*}
        \mathcal{H}(\theta)^\pm \in \mathcal{N}_{\lambda,\varrho} \quad \text{and} \quad \mathcal{I}_{\lambda,\varrho}\big(\mathcal{H}(\theta)^\pm\big) < \frac{1}{N} S^{\frac{N}{2}}, \quad \text{for all } \, \theta \in \mathbb{S}^N.
    \end{equation*}
    This concludes the proof.
\end{proof}
\begin{lemma}\label{lem_H:Rn+N+2}
  Let $\lambda > {\lambda_0}$ and $n \in \mathbb{N} \cup \{0\}$ be the greatest integer such that $\lambda_n < \lambda$. Then there exists an odd continuous map $\widetilde{\mathcal{H}}: \mathbb{R}^{n+N+2} \rightarrow H_0^1(\Omega)$ such that
  \begin{align}\label{eqn_prop_tild_H}
      \lim_{|\xi| \rightarrow \infty} \mathcal{I}_{\lambda, \varrho}( \widetilde{\mathcal{H}}(\xi)) = -\infty \quad \text{and} \quad \sup_{ v \in \widetilde{\mathcal{H}}(\mathbb{R}^{n+N+2}) } \mathcal{I}_{\lambda, \varrho}(v) < \frac{2}{N}S^{\frac{N}{2}}.
  \end{align}
\end{lemma}

\begin{proof}
    We first consider the case where $n \ge 1$. Let $V^- = \mathrm{span}\{e_1,\ldots,e_n\}$ be the eigenspace corresponding to the eigenvalues $\{\lambda_1,\ldots,\lambda_n\}$ of the weighted linear problem \eqref{eqn_ev_pro_en}. Define the unit sphere in this subspace:
    \[
        \mathcal{S}^- := \{ v \in V^- : \| v\| = 1 \}.
    \]
    We claim that there exists $\delta > 0$ such that the quadratic form associated with the linear part of $\mathcal{I}_{\lambda,\varrho}$ is strictly negative in a neighborhood of $\mathcal{S}^-$:
    \begin{equation}\label{eq_claimQ_lam<}
        Q_\lambda(v) := \|v\|^2 - \left(\lambda -\lambda_0\right) |\varrho v|_2^2 < 0 \quad \text{for all } v \in B_\delta(\mathcal{S}^-).
    \end{equation}
    Any $v \in \mathcal{S}^-$ can be expressed as $v = \sum_{j=1}^n \alpha_j e_j$ with $\sum_{j=1}^n \alpha_j^2 \left(\lambda_j -\lambda_0\right) = 1$. It follows that
    \[
        \max_{v \in \mathcal{S}^-} Q_\lambda(v)
        = 1 - \left(\lambda -\lambda_0\right) \min_{v \in \mathcal{S}^-} |\varrho v|_2^2
        = 1 - \frac{\lambda -\lambda_0}{\lambda_n -\lambda_0}.
    \]
    Since $\lambda_n < \lambda$, the ratio is strictly greater than $1$. Thus, we can set $-\varepsilon_0 := 1 - \frac{\lambda -\lambda_0}{\lambda_n -\lambda_0} < 0$, establishing that $Q_\lambda(v) \le -\varepsilon_0$ for all $v \in \mathcal{S}^-$. 
    
    By continuity, this strict negativity persists for small perturbations. Specifically, if $v \in B_\delta(\mathcal{S}^-)$, we can write $v = v_1 + \varphi$ with $v_1 \in \mathcal{S}^-$ and $\|\varphi\| < \delta$. Using the Cauchy--Schwarz inequality, one easily bounds the cross terms to find
    \[
        Q_\lambda(v) \le -\varepsilon_0 + C(\lambda, \lambda_1)\delta + \delta^2.
    \]
    Choosing $\delta > 0$ sufficiently small such that $C(\lambda, \lambda_1)\delta + \delta^2 < \frac{\varepsilon_0}{2}$ confirms the claim \eqref{eq_claimQ_lam<}.

    Let $x \in \Omega$ and choose $r > 0$ such that $\mathbf{B}(x,r) \subset \Omega$. For $0 < r_0 < r$, let $\psi_{r_0}$ denote the minimizer of $\operatorname{cap}_{\mathbf{B}(x,r)}(\mathbf{B}(x,r_0))$.
    Let $x \in \Omega$ and $\mathbf{B}(x,r) \subset \Omega$. For $0 < r_0 < r$, let $\psi_{r_0}$ denote the minimizer of $\operatorname{cap}_{\mathbf{B}(x,r)}(\mathbf{B}(x,r_0))$. 
    {We fix $r_0 > 0$ sufficiently small such that for every $v \in \mathcal{S}^-$, the normalized truncated function $\frac{(1-\psi_{r_0})v}{\|(1-\psi_{r_0})v\|} \in B_\delta(\mathcal{S}^-)$.} 
    We define the linear map $\mathcal{H}_0 : \mathbb{R}^n \to H_0^1(\Omega \setminus \mathbf{B}(x,r_0))$ by
    \[
        \mathcal{H}_0(\xi_1,\ldots,\xi_n) := (1-\psi_{r_0}) \sum_{j=1}^n \xi_j e_j.
    \]
    Because the normalized elements of $\mathcal{H}_0(\mathbb{R}^n \setminus \{0\})$ belong to $B_\delta(\mathcal{S}^-)$, equation \eqref{eq_claimQ_lam<} dictates that the quadratic form $Q_\lambda$ is strictly negative on the entire subspace $\mathcal{H}_0(\mathbb{R}^n)$. Since the nonlinear term in the energy functional is also strictly negative, this yields
    \[
        \sup_{v \in \mathcal{H}_0(\mathbb{R}^n)} \mathcal{I}_{\lambda,\varrho}(v) \le 0.
    \]
    On the other hand, by Lemma \ref{lem_I_l_Sn}, there exists an odd continuous map $\mathcal{H} : \mathbb{S}^N \to H_0^1(\mathbf{B}(x,r_0/2))$ such that 
    \[
        \mathcal{H}(\theta)^\pm \in \mathcal{N}_{\lambda,\varrho} \quad \text{and} \quad \mathcal{I}_{\lambda,\varrho}(\mathcal{H}(\theta)^\pm) < \frac{1}{N} S^{N/2}.
    \]
    By choosing a cutoff function $\eta$ supported in the annulus $\mathbf{B}(x,r_0) \setminus \mathbf{B}(x,r_0/2)$ and projecting a highly concentrated bubble $\mathcal{R}_{\lambda, \varrho}(\eta U_{\varepsilon,0})$, we obtain a function
    \[
        \varphi_0 \in H_0^1\bigl(\mathbf{B}(x,r_0) \setminus \mathbf{B}(x,r_0/2)\bigr) \cap \mathcal{N}_{\lambda,\varrho} \quad \text{with} \quad \mathcal{I}_{\lambda,\varrho}(\varphi_0) < \frac{1}{N} S^{N/2}.
    \]
    We construct an intermediate domain
    \[
        Z := \left(\mathbb{S}^N \times [-1,1] \right) \cup \left(\overline{\mathbb{B}^{N+1}} \times \{-1,1\}\right) \subset \mathbb{R}^{N+2},
    \]
    and extend $\mathcal{H}$ to a continuous map $\mathcal{H}' : Z \to H_0^1(\mathbf{B}(x,r_0))$ by defining it on the boundary faces:
    \[
        \mathcal{H}'(s\theta,t) =
        \begin{cases}
            (1-t)\,\mathcal{H}(\theta)^- + (1+t)\,\mathcal{H}(\theta)^+, & s=1,\\[4pt]
            2s\,\mathcal{H}(\theta)^+ + (1-s)\varphi_0, & t=1,\\[4pt]
            2s\,\mathcal{H}(\theta)^- - (1-s)\varphi_0, & t=-1.
        \end{cases}
    \]
    Extending $\mathcal{H}'$ radially yields $\mathcal{H}^0 : \mathbb{R}^{N+2} \to H_0^1(\mathbf{B}(x,r_0))$ given by $\mathcal{H}^0(\zeta z) := \zeta\,\mathcal{H}'(z)$ for $\zeta \ge 0$. Because $\mathcal{H}'$ maps into $\mathcal{N}_{\lambda, \varrho}$ where energy is bounded below $\frac{1}{N}S^{N/2}$, and taking into account the interaction of disjoint supports, $\mathcal{H}^0$ is an odd continuous map satisfying
    \[
        \lim_{|\xi|\to\infty} \mathcal{I}_{\lambda,\varrho}(\mathcal{H}^0(\xi)) = -\infty  \quad \text{and} \quad  \sup_{v \in \mathcal{H}^0(\mathbb{R}^{N+2})} \mathcal{I}_{\lambda,\varrho}(v) < \frac{2}{N} S^{N/2}.
    \]
    Finally, for $n \ge 1$, we join the low-energy space and the bubble topology by defining
    \[
        \widetilde{\mathcal{H}}(\xi_1,\xi_2) := \mathcal{H}_0(\xi_1) + \mathcal{H}^0(\xi_2), \qquad \xi_1 \in \mathbb{R}^n,\ \xi_2 \in \mathbb{R}^{N+2}.
    \]
    In standard Euclidean problems, one often achieves energy decoupling between the linear eigenspace and the topological bubbles simply by translating the bubbles infinitely far away. However, in our hyperbolic framework, the conformal weight $\varrho(x)$ strictly prevents such global shifts. To bypass this broken translation invariance, our use of the capacity minimizer $\psi_{r_0}$ rigorously enforces that $\mathcal{H}_0$ and $\mathcal{H}^0$ have exactly disjoint supports (the former vanishes entirely on $\mathbf{B}(x,r_0)$ while the latter is strictly supported within it). Consequently, their energies decouple perfectly without any need for spatial translation: $\mathcal{I}_{\lambda, \varrho}(\mathcal{H}_0 + \mathcal{H}^0) = \mathcal{I}_{\lambda, \varrho}(\mathcal{H}_0) + \mathcal{I}_{\lambda, \varrho}(\mathcal{H}^0)$. Since the energy of the linear component $\mathcal{H}_0$ is non-positive, the combined map strictly inherits the upper bound of the bubbles, yielding $\sup \mathcal{I}_{\lambda, \varrho}(\widetilde{\mathcal{H}}) < \frac{2}{N}S^{N/2}$. When $n=0$, we simply take $\widetilde{\mathcal{H}} := \mathcal{H}^0$.
\end{proof}
\subsection{Proof of the Multiplicity Result}
In this final section, we synthesize our geometric energy estimates and the equivariant category framework to prove our main multiplicity result. We begin by establishing strict upper bounds on the minimax energy levels $c_k$, ensuring they remain strictly below the compactness threshold. 

\begin{proposition}\label{prop_main_thm}
We have the following strict upper bounds for the minimax values $c_k$:
\begin{enumerate}    
    \item If $\lambda_n < \lambda < \lambda_{n+1}$ for some $n \ge 1$, then
    \begin{align}\label{eq_c_N+2<}
            c_{N+2} < \frac{2}{N} S^{\frac{N}{2}}.
    \end{align}
    \item If $\lambda_0 < \lambda < \lambda_1$, then
    \begin{align}\label{eq_c_N+1<}
                c_{N+1} < \frac{2}{N} S^{\frac{N}{2}}.        
    \end{align}
    \item If $\lambda_n < \lambda = \lambda_{n+1} = \cdots = \lambda_{n+m} < \lambda_{n+m+1}$ with $m < N+2$, then
    \begin{align}\label{eq_c_N+2-m<}
        c_{N+2-m} < \frac{2}{N} S^{\frac{N}{2}}. 
    \end{align}
\end{enumerate}
\end{proposition}

\begin{proof}
Let $n \ge 0$ be the largest integer such that $\lambda_n < \lambda$. By Lemma \ref{lem_H:Rn+N+2}, there exists an odd continuous map
\[
    \widetilde{\mathcal{H}} : \mathbb{R}^{n+N+2} \to H.
\]
From the energy estimate \eqref{eqn_prop_tild_H}, we obtain the uniform bound
\begin{equation}\label{eq_def_c_inf}
    c_\infty := 
    \sup_{v \in \widetilde{\mathcal{H}}(\mathbb{R}^{n+N+2})} \mathcal{I}_{\lambda,\varrho}(v)
    < \frac{2}{N} S^{\frac{N}{2}}.
\end{equation}
Let us define the integer $k$ as the relative category evaluated at this maximum energy:
\[
    k := \gamma_{D_\lambda} \bigl(\mathcal{I}_{\lambda,\varrho}^{c_\infty} \cup D_\lambda \bigr).
\]
By the definition of the relative category (Definition \ref{de_rel_cat}), there exists an open symmetric covering $U_0, U_1, \dots, U_k$ of the set $\mathcal{I}_{\lambda,\varrho}^{c_\infty} \cup D_\lambda$, together with odd continuous maps
\[
    \chi_0 : U_0 \to D_\lambda, \qquad \chi_j : U_j \to \{-e_{n+j}, e_{n+j}\} \quad \text{for } j=1,\dots,k.
\]
By Tietze’s extension theorem, $\chi_0$ extends to a global odd continuous map $\widetilde{\chi}_0 : H \to H$.

\medskip
\textbf{Case (i):} $\lambda_n < \lambda < \lambda_{n+1}$ for some $n \ge 1$.
Let us define the pre-image set
\[
    \mathcal{O} := \{\xi \in \mathbb{R}^{n+N+2} : \|\widetilde{\chi}_0(\widetilde{\mathcal{H}}(\xi))\| \le r_\lambda\},
\]
where $r_{\lambda}$ is the radius from \eqref{eq_rel_I,d} satisfying
\begin{equation}\label{eq_I>2dl_rl}
    \mathcal{I}_{\lambda, \varrho}(v) \ge 2 d_\lambda \quad \text{for every } v \in V^+ \text{ with } \|v\| = r_\lambda.
\end{equation}
Since both $\widetilde{\chi}_0$ and $\widetilde{\mathcal{H}}$ are odd functions, $\mathcal{O}$ is a symmetric neighborhood of the origin. Moreover, since $\lim_{|\xi| \to \infty} \mathcal{I}_{\lambda, \varrho}(\widetilde{\mathcal{H}}(\xi)) = -\infty$ and elements mapped to $D_\lambda$ cannot have arbitrarily negative energy, $\mathcal{O}$ is bounded. 

We set $V_j := (\widetilde{\mathcal{H}}^{-1} U_j) \cap \partial \mathcal{O}$ for $j = 0, 1, \ldots, k$. The boundary condition forces $\|\widetilde{\chi}_0(\widetilde{\mathcal{H}}(\xi))\| $ $= r_\lambda$. From \eqref{eq_I>2dl_rl} and the fact that $\widetilde{\chi}_0$ maps into $D_\lambda = \mathcal{I}_{\lambda, \varrho}^{d_\lambda}$, the energy is too low to lie in $V^+$. Thus,
$$
\widetilde{\chi}_0(\widetilde{\mathcal{H}}(V_0)) \subset \{ v \in H : \|v\| = r_\lambda \} \setminus V^+.
$$
Composing the map $\widetilde{\chi}_0 \circ \widetilde{\mathcal{H}}\mid_{V_0}$ with the orthogonal projection $H \to V^-$ yields a well-defined, non-vanishing odd continuous map
\[
{\chi^0} : V_0 \to V^- \setminus \{0\}.
\]
Take an even partition of unity $\{\pi_0, \pi_1, \ldots, \pi_k\}$ subordinated to the covering $\{V_0, V_1, \ldots, V_k\}$ of $\partial \mathcal{O}$. We construct the combined map $\chi$ by
\begin{align}\label{eq_def_chIprop}
\chi(\xi) = \pi_0(\xi) {\chi}^0(\xi) + \sum_{j=1}^k \pi_j(\xi) \chi_j( \widetilde{\mathcal{H}} (\xi)).
\end{align}
Because the targets of $\chi_j$ are basis vectors spanning mutually orthogonal dimensions, $\chi$ maps into $\mathrm{span}\{e_1, \ldots, e_{n+k}\} \cong \mathbb{R}^{n+k}$. Furthermore, $\chi$ is an odd continuous map such that $\chi(\xi) \neq 0$ for every $\xi \in \partial \mathcal{O}$.

The boundary $\partial \mathcal{O}$ bounds a symmetric neighborhood of the origin in $\mathbb{R}^{n+N+2}$, making it homeomorphic to the sphere $\mathbb{S}^{n+N+1}$. The target space excluding the origin is homotopy equivalent to $\mathbb{S}^{n+k-1}$. By the Borsuk-Ulam theorem, there cannot exist an odd continuous map from $\mathbb{S}^{n+N+1}$ to $\mathbb{S}^{n+k-1}$ unless the dimension of the domain is less than or equal to the dimension of the target. Thus, we must have
$$n + N + 1 \le n + k - 1 \implies k \ge N + 2.$$
This topological dimension bound is equivalent to the category estimate $\gamma_{D_\lambda}(\mathcal{I}_{\lambda, \varrho}^{c_{\infty}} \cup D_\lambda) \geq N+2$. Since $c_{N+2}$ is the infimum of all values $c$ for which the relative category is at least $N+2$, we deduce from \eqref{eq_def_c_inf} the strict upper bound $c_{N+2} \leq c_{\infty} < \frac{2}{N} S^{\frac{N}{2}}$.

\medskip
\textbf{Case (ii):} $\lambda_0 < \lambda < \lambda_1$. 
Due to the variational characterization of $\lambda_1$ from $\eqref{eqn_prop_lam1}$, we have the strict inequality $(\lambda -\lambda_0)|\varrho v|_2^2 < (\lambda_1 -\lambda_0) |\varrho v|_2^2 \leq \|v\|^2$. This guarantees that the quadratic part of the functional is strictly positive: $\|v\|^2 - (\lambda -\lambda_0) |\varrho v|_{2}^2 > 0$ for all $v \in H \setminus \{0\}$. Consequently, the radial domain $\mathcal{V}_{\lambda, \varrho}$ is the entire space $H \setminus \{0\}$.

Let us define the ground state energy on the Nehari manifold
\begin{align}\label{de_c_0}
    c_0 := \inf_{v \in \mathcal{N}_{\lambda, \varrho}} \mathcal{I}_{\lambda, \varrho}(v).
\end{align}
Stapelkamp \cite[Page 57]{Sta03} demonstrated that for $\lambda \in \left(\lambda_0, \lambda_1\right)$, problem $\eqref{eq_trns_BNP_v}$ admits at least one positive solution. This positive solution is exactly the minimizer $v_0$ that attains $c_0$ on $\mathcal{N}_{\lambda, \varrho}$. Thus, it follows that for all $v \in \mathcal{N}_{\lambda, \varrho}$:
 \begin{align}\label{eqn_I_l>c0}
     \mathcal{I}_{\lambda, \varrho}(v) = \frac{1}{N}|v|_{2^*}^{2^*} \geq c_0 > 0.
 \end{align} 
We define the sign-changing Nehari set
$$
\mathcal{E}_{\lambda, \varrho} := \{ v \in \mathcal{N}_{\lambda, \varrho} : v^+ \in \mathcal{N}_{\lambda, \varrho} \ \text{and} \ v^- \in \mathcal{N}_{\lambda, \varrho} \}.
$$
Assume, for contradiction, that the distance between $\mathcal{E}_{\lambda, \varrho}$ and the sign-definite set $P \cup (-P)$ is zero. Then there exists a sequence $\{v_j\}_{j \in \mathbb{N}} \subset \mathcal{E}_{\lambda, \varrho}$ such that
\begin{align}\label{dist_vj_p_0}
\lim_{j \rightarrow \infty}\mathrm{dist}(v_j, P \cup (-P)) = 0.    
\end{align}
This implies that a subsequence either converges toward $P$ or toward $-P$:
\begin{enumerate}[(I)]
    \item \textit{If $v_j$ converges toward $P$:} Recalling the Sobolev estimate $\eqref{eqn_v-<d(u,p)}$, we have
    $$
    |v^-_j|_{2^*} \leq S^{-1/2} \mathrm{dist}(v_j, P) \rightarrow 0 \quad \text{as } j \to \infty.
    $$
    However, since $v_j \in \mathcal{E}_{\lambda, \varrho}$, its negative part $v^-_j$ lies strictly on the Nehari manifold. By $\eqref{eqn_I_l>c0}$, this requires $\frac{1}{N} |v^-_j|_{2^*}^{2^*} \geq c_0 > 0$, yielding an immediate contradiction.
    \item \textit{If $v_j$ converges toward $-P$:} By \eqref{eq_vj+<d(vj,P)},
    $$
    |v^+_j|_{2^*} \leq S^{-1/2} \mathrm{dist}(v_j, -P) \rightarrow 0, \quad \text{as } j \to \infty,
    $$
    which again contradicts the strict energy lower bound $\eqref{eqn_I_l>c0}$.
\end{enumerate}
Thus, assumption \eqref{dist_vj_p_0} is false, and there exists a uniform constant $\alpha_0 > 0$ such that $\mathrm{dist}(v, P \cup (-P)) \geq \alpha_0$ for every $v \in \mathcal{E}_{\lambda, \varrho}$. By choosing $\alpha < \alpha_0$ in the definition \eqref{de_D_lam} of $D_\lambda$, we strictly separate the target set, obtaining $D_\lambda \cap \mathcal{N}_{\lambda, \varrho} \subset \mathcal{N}_{\lambda, \varrho} \setminus \mathcal{E}_{\lambda, \varrho}$. It follows by standard topological arguments (see \cite[Lemma 2.5]{CCN97}) that $\mathcal{N}_{\lambda, \varrho} \setminus \mathcal{E}_{\lambda, \varrho}$ consists of exactly two connected components, say $\mathcal{W}$ and $-\mathcal{W}$, thereby admitting a continuous odd map $\Psi : \mathcal{N}_{\lambda, \varrho} \setminus \mathcal{E}_{\lambda, \varrho} \to \{-e_{k+1}, e_{k+1}\}$.

Since in this regime the Nehari manifold $\mathcal{N}_{\lambda, \varrho}$ is radially diffeomorphic to the unit sphere in $H$, we let $\mathcal{C}_0$ be the bounded connected component of $H \setminus \mathcal{N}_{\lambda, \varrho}$ containing the origin. Set
$$
\mathcal{O} := \{ \xi \in \mathbb{R}^{N+2} : \widetilde{\chi}_0(\widetilde{\mathcal{H}}(\xi)) \in \overline{\mathcal{C}_0} \}.
$$
As $\lim_{|\xi| \to \infty} \mathcal{I}_{\lambda, \varrho}(\widetilde{\mathcal{H}}(\xi)) = -\infty$, $\mathcal{O}$ is a bounded symmetric neighborhood of the origin. Setting $V_j := (\widetilde{\mathcal{H}}^{-1} U_j) \cap \partial \mathcal{O}$ for $j = 0, \ldots, k$, we define
$$
\chi^0 := \Psi \circ \widetilde{\chi}_0 \circ \widetilde{\mathcal{H}} : V_0 \rightarrow \{ -e_{k+1}, e_{k+1}\}.
$$
Using a construction directly analogous to $\eqref{eq_def_chIprop}$, we obtain an odd continuous map $\chi : \partial \mathcal{O} \to \mathrm{span}\{e_1, \ldots, e_{k+1}\} \cong \mathbb{R}^{k+1}$ that does not vanish on $\partial \mathcal{O}$. Here, the domain boundary is $\mathbb{S}^{N+1}$ and the target is $\mathbb{R}^{k+1} \setminus \{0\} \simeq \mathbb{S}^k$. Applying the Borsuk-Ulam theorem dictates that $N+1 \le k$, or $k \geq N+1$. This establishes the category estimate $\gamma_{D_\lambda}(\mathcal{I}_{\lambda, \varrho}^{c_{\infty}} \cup D_\lambda) \geq N+1$, which guarantees $c_{N+1} \leq c_{\infty} < \frac{2}{N} S^{\frac{N}{2}}$.

\textbf{Case (iii):} $\lambda = \lambda_{n+1} = \cdots = \lambda_{n+m}$ is an eigenvalue of multiplicity $m < N+2$. 
The proof structurally mimics \textbf{Case (i)}, but the map ${\chi}^0$ projects onto a higher-dimensional target space: ${\chi}^0 : V_0 \to \mathrm{span}\{e_1, \ldots, e_{n+m}\} \setminus \{0\}$. The resulting aggregated odd continuous map is
\[
\chi : \partial \mathcal{O} \to \mathrm{span}\{e_1, \ldots, e_{n+m+k}\} \setminus \{0\} \cong \mathbb{R}^{n+m+k} \setminus \{0\}.
\]
Applying Borsuk-Ulam to the dimensions yields $n + N + 1 \le n + m + k - 1$, which simplifies to $k \ge N + 2 - m$. This category estimate translates directly to the bound $c_{N+2-m} < \frac{2}{N} S^{\frac{N}{2}}$.
\end{proof}

With these bounds established, we now utilize the minimax energy levels $c_k$ to count the distinct pairs of critical points. This allows us to prove Theorem \ref{th_MBNP_Hn>Rn}, from which our primary geometric result, Theorem \ref{th_MBNP_Hn}, directly follows.

\begin{proof}[\textbf{Proof of Theorem \ref{th_MBNP_Hn>Rn}}]
We assume throughout that the critical set $K_c$ is finite for every $c < \frac{2}{N} S^{N/2}$; otherwise, if any $K_c$ is infinite, the theorem is immediately satisfied. We structure the multiplicity counting according to the eigenvalue spectrum.

\textbf{Case (i):} $\lambda_n < \lambda < \lambda_{n+1}$ for some $n \ge 1$.  
The minimax levels are monotonically non-decreasing: $c_1 \le c_2 \le \cdots \le c_k$. Using the bound $c_{N+2} < \frac{2}{N}S^{N/2}$ from Proposition \ref{prop_main_thm}, if $c_k = c_{k+1}$ for any $k \le N+1$, Lemma \ref{lem_inf_case_Kc} guarantees that $K_{c_k}$ is infinite, and we are done. Thus, we may assume the strict inequalities
\begin{align}\label{eq_c1<c2<S}
c_1 < c_2 < \cdots < c_{N+2} < \frac{2}{N}S^{N/2}.    
\end{align}
If all of these levels lie above the single-bubble threshold (i.e., $c_k \ge \frac{1}{N}S^{N/2}$ for all $1 \le k \le N+1$), then Corollary \ref{cor_cri_val_I}\eqref{it_c_k_or_ck-} dictates that for each $k$, either $c_k$ or $c_k - \frac{1}{N}S^{N/2}$ is a true critical value. Because the $c_k$ sequence is strictly increasing, this generates at least $N+1$ distinct critical points, more than proving the claim. 

Otherwise, there exists an integer $k_0 \le N+1$ defining a crossover
\[
c_{k_0} < \frac{1}{N}S^{N/2} \le c_{k_0+1}.
\]
By Corollary \ref{cor_cri_val_I}\eqref{it_c_k<SN/2}, the functional $\mathcal{I}_{\lambda,\varrho}$ possesses exact critical points for all values below the threshold, and shadow critical points (at $c_k$ or $c_k - \frac{1}{N}S^{N/2}$) above it. This yields at least
\[
k_{\min} := \max \left\{k_0, (N+2) - (k_0+1)\right\}
\]
distinct non-trivial critical pairs. By elementary properties of the maximum, $k_{\min} \ge \frac{N+1}{2}$, which verifies the multiplicity claim in Theorem \ref{th_MBNP_Hn>Rn}(i). Furthermore, the energy estimate \eqref{eq_nabu2<2S_Rn} follows directly from the fact that all relevant critical points reside on the Nehari manifold where $\mathcal{I}_{\lambda,\varrho}(v) < \frac{2}{N}S^{N/2}$.

\textbf{Case (ii):} $\lambda_0 < \lambda < \lambda_1$.  
Recall the ground state energy $c_0$ from \eqref{de_c_0}. From the bubble estimate \eqref{eqn_I<1/N_S}, we know $0 < c_0 < \frac{1}{N}S^{N/2}$. 

We claim that $K_{c_0} \subset P \cup (-P)$. Indeed, if $v \in K_{c_0}$ were a sign-changing critical point, then both its positive and negative components $v^+$ and $v^-$ would belong to the Nehari manifold $\mathcal{N}_{\lambda,\varrho}$. Since $\mathcal{I}_{\lambda,\varrho}$ is additive over disjoint supports, we would have
\[
\mathcal{I}_{\lambda,\varrho}(v) = \mathcal{I}_{\lambda,\varrho}(v^+) + \mathcal{I}_{\lambda,\varrho}(v^-) \ge 2c_0,
\]
which contradicts $\mathcal{I}_{\lambda,\varrho}(v) = c_0 > 0$. Thus, all critical points at the ground state energy are strictly sign-definite. 

Crucially, because $D_\lambda$ includes an $\alpha$-neighborhood of $P \cup (-P)$, the ground state critical set $K_{c_0}$ lies strictly within the interior of the invariant set $D_\lambda$. Since the relative category minimax value $c_1$ requires topological coverings of sets traversing outside $D_\lambda$, it follows that $c_1$ must strictly exceed the energy of elements confined inside $D_\lambda$. Therefore, $c_0 < c_1$. 

As in \textbf{Case (i)}, utilizing \eqref{eq_c_N+1<}, we can assume a strictly increasing chain of minimax values
\[
c_0 < c_1 < c_2 < \cdots < c_{N+1} < \frac{2}{N}S^{N/2}.
\]
By Corollary \ref{cor_cri_val_I}, there exists an integer $\tilde{k}_0 \le N+1$ marking the compactness threshold
\[
c_{\tilde{k}_0} < \frac{1}{N}S^{N/2} \le c_{\tilde{k}_0 + 1}.
\]
Accounting for $c_0$ alongside the higher levels, the functional admits at least
\[
\tilde{k}_{\min} := \max\left\{\tilde{k}_0 + 1, (N+2) - (\tilde{k}_0 + 1)\right\}
\]
non-trivial critical pairs. Since $\tilde{k}_{\min} \ge \frac{N+2}{2}$, the multiplicity bound for Theorem \ref{th_MBNP_Hn>Rn}(ii) is established.

\textbf{Case (iii):} $\lambda = \lambda_{n+1} = \cdots = \lambda_{n+m}$ is an eigenvalue of multiplicity $m < N+2$. 
The counting logic mimics \textbf{Case (i)}, but is restricted by the shorter sequence of valid minimax levels due to the multiplicity $m$. By \eqref{eq_c_N+2-m<}, the chain of strict inequalities is reduced to length $N+2-m$. This reduction shifts the maximum counting to $\max\{k_0, (N+2-m) - k_0\} \ge \left\lceil \frac{N+2-m}{2} \right\rceil$, yielding at least $\frac{N+1-m}{2}$ pairs, which completes the proof.
\end{proof}
\section*{Concluding Remarks and Future Directions}

We conclude this article by highlighting a few geometric and analytical aspects of our results, as well as outlining some natural open problems that arise from this work.

\begin{enumerate}
\item \textbf{Nodal Properties of the Solutions:}
    We established that ground state solutions (at energy $c_0$) are strictly sign-definite. Because higher minimax levels $c_k$ ($k \ge 1$) are constructed outside the neighborhood $D_\lambda$ containing the positive and negative cones, the corresponding higher-energy solutions must change sign. A natural open question is whether one can bound the number of nodal domains for these solutions using the topological index $k$.
    %, akin to Courant's nodal domain theorem.
    \item \textbf{The Lower Bound for $\lambda$:}
    Our requirement $\lambda > \frac{N(N-2)}{4}$ is geometrically intrinsic to the problem. While the bottom of the $L^2$-spectrum of $-\Delta_{\bn}$ on the whole space is $\frac{(N-1)^2}{4}$, the constant $\frac{N(N-2)}{4}$ arises directly as the conformal shift associated with the scalar curvature of the hyperbolic space. Below this critical conformal threshold, the transformed Euclidean operator $-\Delta - (\lambda - \frac{N(N-2)}{4})\varrho^2$ may lose its strict coercivity, fundamentally altering the topological structure of the Nehari manifold. Finding multiple solutions for $\lambda \le \frac{N(N-2)}{4}$ via alternative linking structures remains a non-trivial open problem.

\end{enumerate}

\section*{Acknowledgments}
 SG thanks the Anusandhan National Research Foundation (ANRF), India, for support under the ARG-MATRICS Grant No. ANRF/ARGM/2025/001570/MTR. TR is supported by the FWO Odysseus 1 grant G.0H94.18N: Analysis and Partial Differential Equations, the Methusalem program of the Ghent University Special Research Fund (BOF), (TR Project title: BOFMET2021000601). TR is also supported by a BOF postdoctoral fellowship at Ghent University BOF24/PDO/025.

\bibliographystyle{alphaurl}
\bibliography{Ref_MSBNP_on_Hn}
\end{document}